\documentclass[leqno,11pt]{amsart}
\usepackage{amsmath,amssymb}
\usepackage{color}
\usepackage{comment}
 \usepackage[pagewise]{lineno}

\allowdisplaybreaks

\usepackage{enumerate}
\usepackage{color,comment}
\usepackage{mathtools}
\theoremstyle{plain}
\newtheorem{theorem}{Theorem}[section]
\newtheorem{lemma}[theorem]{Lemma}
\newtheorem{proposition}[theorem]{Proposition}
\newtheorem{corollary}[theorem]{Corollary}
\theoremstyle{definition}

\newtheorem{remark}[theorem]{Remark}

\newtheorem{example}[theorem]{Example}

\def\rn{\mathbb R\sp n}
\def\Rn{\mathbb R\sp n}
\def\R{\mathbb R}
\def\N{\mathbb N}
\def\hra{\rightarrow}
\def\M{\mathcal M}

\def\opt{\operatorname{opt}}

\newtoks\by
\newtoks\paper
\newtoks\book
\newtoks\jour
\newtoks\yr
\newtoks\pages
\newtoks\vol
\newtoks\publ
\newtoks\eds
\newtoks\proc
\def\ota{{\hbox{???}}}
\def\cLear{\by=\ota\paper=\ota\book=\ota\jour=\ota\yr=\ota
\pages=\ota\vol=\ota\publ=\ota}
\def\endpaper{\the\by, \textit{\the\paper},
{\the\jour} \textbf{\the\vol} (\the\yr), \the\pages.\cLear}
\def\endbook{\the\by, \textit{\the\book}, \the\publ.\cLear}
\def\endprep{\the\by, \textit{\the\paper}, \the\jour.\cLear}
\def\endproc{\the\by, \textit{\the\paper}, \the\publ, \the\pages.\cLear}
\def\name#1#2{#1 #2}
\def\et{ and }
\numberwithin{equation}{section}

\begin{document}

\title[Sharp Sobolev type embeddings on the entire Euclidean space]
{Sharp Sobolev type embeddings \\ on the entire Euclidean space}

{\date\today}
\author {Angela Alberico, Andrea Cianchi, Lubo\v s Pick and Lenka Slav\'ikov\'a}

\address{Istituto per le Applicazioni di Calcolo ``M. Picone''\\
Consiglio Nazionale delle Ricerche \\
Via Pietro Castellino 111\\
80131 Napoli\\
Italy} \email{a.alberico@iac.cnr.it}

\address{Dipartimento di Matematica \lq\lq U. Dini"\\
Universit\`a di Firenze\\
Piazza Ghiberti 27\\
50122 Firenze\\
Italy} \email{cianchi@unifi.it}

\address{Department of Mathematical Analysis\\
Faculty of Mathematics and Physics\\
Charles University\\
Sokolovsk\'a~83\\
186~75 Praha~8\\
Czech Republic} \email{pick@karlin.mff.cuni.cz}

\address{Department of Mathematical Analysis\\
Faculty of Mathematics and Physics\\
Charles University\\
Sokolovsk\'a~83\\
186~75 Praha~8\\
Czech Republic}
\curraddr{Department of Mathematics\\
University of Missouri\\
Columbia, MO 65211\\
USA} \email{slavikoval@missouri.edu}


\iffalse
\thanks{
This research was partly supported by the Research Project of
Italian Ministry of University and Research (MIUR)  2012TC7588
``Elliptic and parabolic partial differential equations: geometric
aspects, related inequalities, and applications" 2012,  by  GNAMPA
of Italian INdAM (National Institute of High Mathematics),  by the
grant P201-13-14743S of the Grant Agency of the Czech Republic,  by
the grant SVV-2013-267316 and the Mobility Foundation of the Charles
University in Prague.} \fi

\begin{abstract}
A comprehensive approach to Sobolev-type embeddings, involving
arbitrary rearrangement-invariant norms  on the entire Euclidean
space $\rn$, is offered. In particular, the optimal target space in
any such embedding is exhibited. Crucial in our analysis is a new
reduction principle for the relevant embeddings, showing their
equivalence to a couple of considerably simpler one-dimensional
inequalities. Applications to  the classes of  the
 Orlicz-Sobolev and the Lorentz-Sobolev spaces are also presented.
These contributions fill in  a gap in the existing literature, where
sharp results in such a general setting are only available   for
domains of finite measure.

\iffalse show various novelties compared with the case of domains of
finite measure, the only case for which results in this generality
are available in the existing literature.

We study the question when a Sobolev space $W\sp m X(\rn)$,
containing weakly differentiable functions defined on the entire
Euclidean space $\rn$ whose $m$-th gradient (with fixed $m\in\N$)
belongs to a given rearrangement-invariant Banach function space
$X(\rn)$, is continuously embedded into another such space,
$Y(\rn)$. We find necessary and sufficient conditions for such an
embedding in terms of certain Hardy-type operators. With the help of
this result, we characterize the optimal range space when the domain
space is fixed. Results of this type had been known earlier for
Sobolev spaces on domains of finite measure. \fi
\end{abstract}

\maketitle

\section{Introduction}\label{S:intro}

An embedding theorem of Sobolev type amounts to a statement
asserting that a certain  degree of integrability of the (weak)
derivatives of a function entails extra integrability of the
function itself. A basic formulation concerns  the space of weakly
differentiable functions $W^{1,p}(\rn )$, endowed with the norm
$\|u\|_{W^{1,p}(\rn )}= \|u\|_{L^p(\rn)} + \|\nabla u\|_{L^p(\rn)}$.
Membership of $u$ in $W^{1,p}(\rn)$ guarantees that $u \in L^{\frac
{np}{n-p}}(\rn)$ if $1\leq p<n$, or $u \in L^\infty (\rn)$, if
$p>n$. Since $\frac {np}{n-p} >p$, this is locally a stronger
property than the a priori assumption that $u \in L^p(\rn)$, but it
is weaker, and hence does not add any further information, near
infinity.
\\ The same phenomenon occurs in a higher-order version of this result for $W^{m,p}(\rn)$, the Sobolev space
of those $m$-times weakly differentiable functions $u$ such that the
norm
\begin{equation}\label{sobnorm}
\|u\|_{W^{m,p}(\rn )}= \sum _{k=0}^m\|\nabla ^k u\|_{L^p(\rn)}
\end{equation}
is finite. Here, $\nabla ^k u$ denotes the vector of all derivatives
of $u$ of order $k$, and, in particular, $\nabla ^1u$ stands for
$\nabla u$ and $\nabla ^0 u$ for $u$. Indeed, one has that
\begin{equation}\label{classical}
W^{m,p}(\rn ) \to \begin{cases} L^{\frac{np}{n-mp}}(\rn) \cap
L^p(\rn) & \quad \hbox{if $\;1 \leq m <n$ and $1 \leq p< \frac nm$,}
%\\
%\exp_{[\frac nm -1]} L^{\frac n{n-m}}(\rn ) \cap L^{\frac nm}(\rn )
%& \quad \hbox{if $1<p= \frac nm$,}
\\
L^\infty(\rn)\cap L^p(\rn ) & \quad \hbox{if either $m \geq n$, or
$1 \leq m < n$ and $p> \frac nm$.}
\end{cases}
\end{equation}
%where, for $k \in \N \cup \{0\}$ and $\alpha >0$, $\exp_{k}
%L^{\alpha}(\rn )$ denotes the Orlicz space associated with the Young
%function defined as
%$$A(t) = e^{t^\alpha}- \sum _{j=0}^k \tfrac {t^{\alpha j}}{j!} \quad \hbox{for $t \geq 0$.}$$
%\\ No optimal Lebesgue target space exists.
%On the other hand, an
%optimal target space exists in the framework of Orlicz spaces.
%Indeed,
Embedding \eqref{classical} is still valid if
 $\rn$ is replaced with an open subset $\Omega$ with
finite measure and regular boundary. The intersection with
$L^p(\Omega)$ is irrelevant in the target spaces of the resulting
embedding, and $L^{\frac{np}{n-mp}}(\Omega)$, or $L^\infty (\Omega)$
are  optimal  among all  Lebesgue target spaces.
 By contrast, no optimal Lebesgue target space exists in
\eqref{classical}.
\par The existence of  optimal target spaces in the Sobolev
embedding  for $W^{m,p}(\rn)$ can be restored if the class of
admissible targets is enlarged, for instance, to all Orlicz spaces.
This class allows to describe  different degrees of integrability --
not necessarily of power type -- locally and near infinity. The use
of Orlicz spaces naturally emerges in
%an appropriate description of
the borderline missing case in \eqref{classical}, corresponding to
the exponents $1<p=\frac nm$, and enables one to cover the full
range of exponent $m$ and $p$. The resulting embedding takes the
form
\begin{equation}\label{E:B}
W^{m,p}(\rn ) \to L^B(\rn),
\end{equation}
where $L^B(\rn)$ is the Orlicz space associated with a Young
function $B$ obeying
\begin{equation*}%\label{Binfty}
B(t) \approx
\begin{cases} t^{\frac {np}{n-mp}} & \quad \hbox{if $\;\;1\leq p < \frac nm$}
\\ e^{t^{\frac{n}{n-m}}} & \quad \hbox{if $\;\;1<p= \frac nm$}
\\ \infty & \quad \hbox{otherwise
%if either $m \geq n$, or $1 \leq m < n$ and $p> \frac nm$
}
\end{cases} \quad \quad\hbox{near infinity, } \quad \qquad  B(t) \approx t^p \quad \hbox{near zero.}
\end{equation*}
%near infinity, and
%\begin{equation}\label{B0} B(t) \approx t^p \quad \hbox{near $0$.}
%\end{equation}
Here, $\approx$ denotes equivalence in the sense of Young functions
-- see Section \ref{S:preliminaries}. \iffalse
\begin{equation}\label{Binfty} B(t) \approx
\left\{ \begin{aligned}
        \vf
t^{\frac {np}{n-mp}} & \quad \hbox{if $1\leq p < \frac nm$}
\\ e^{t^{\frac{n}{n-m}}} & \quad \hbox{if $1<p= \frac nm$}
\\ \infty & \quad \hbox{otherwise
%if either $m \geq n$, or $1 \leq m < n$ and $p> \frac nm$
}
\end{cases} \quad \hbox{near infinity, \quad and} \qquad  B(t) \approx t^p \quad \hbox{near zero.}
\end{equation}
%near infinity, and
%\begin{equation}\label{B0} B(t) \approx t^p \quad \hbox{near $0$.}
%\end{equation}

\left.
    \begin{aligned}
        \vf W^{m}_0L^{\frac{nq}{n+mq}}(\log L)^{\frac{n\alpha}{n+mq}}(\Omega) \\
        \vf W^{m}_0L(\log L)^{\alpha(1-\frac{m}{n})}(\Omega) \\
        \vf W^{m,1}_0(\Omega)
    \end{aligned}
    \right\}
\fi

%
%\begin{equation}\label{B3}
%B(t) \approx \begin{cases} t^{p} \quad \hbox{near $0$}
%\\ \infty \quad\hbox{near infinity}
%\end{cases}
%\quad \hbox{if either $m \geq n$, or $1 \leq m < n$ and $p> \frac
%nm$.}
%\end{equation}
%
%
%
%\begin{equation}\label{B1}
%B(t) \approx \begin{cases} t^p \quad \hbox{near $0$}
%\\ t^{\frac{np}{n-mp}} \quad\hbox{near infinity}
%\end{cases}
%\quad \hbox{if $1 \leq p< \frac nm$,}
%\end{equation}
%\begin{equation}\label{B2}
%B(t) \approx \begin{cases} t^{\frac nm} \quad \hbox{near $0$}
%\\ e^{t^{\frac{n}{n-m}}} \quad\hbox{near infinity}
%\end{cases}
%\quad \hbox{if $1<p= \frac nm$,}
%\end{equation}
%\begin{equation}\label{B3}
%B(t) \approx \begin{cases} t^{p} \quad \hbox{near $0$}
%\\ \infty \quad\hbox{near infinity}
%\end{cases}
%\quad \hbox{if either $m \geq n$, or $1 \leq m < n$ and $p> \frac
%nm$.}
%\end{equation}
Actually, embedding \eqref{E:B} is equivalent to \eqref{classical}
if either $1\leq p < \frac nm$, or $m \geq n$, or $1 \leq m < n$ and
$p> \frac nm$. The case when $1<p = \frac nm$ is a consequence of a
result of
%can be found in
%\cite{Ozawa}, and extends to $\rn$ a result by
\cite{Poh, St, Yud}, which, for $m=1$, is also contained in
\cite{Tru}.
\\ The space $L^B(\rn)$ is the optimal (i.e. smallest) target in \eqref{E:B} among
all Orlicz spaces. However, embedding \eqref{E:B} can  still be
improved when either $1\leq p < \frac nm$, or $1<p =\frac nm$,
provided that the class of admissible target spaces is further
broadened. For instance, if $1\leq p < \frac nm$, then
\begin{equation}\label{lorentz}
W^{m,p}(\rn ) \to L^{\frac{np}{n-mp},p}(\rn) \cap L^p(\rn),
\end{equation}
where $L^{\frac{np}{n-mp},p}(\rn)$ is a Lorentz space. The
intersection space in \eqref{lorentz} is the best possible among all
rearrangement-invariant spaces. A parallel result can be shown to
hold in the limiting situation corresponding to $1<p =\frac nm$, and
involves spaces of Lorentz-Zygmund type.
\par This
follows as a special case  of the results of the present paper,
whose purpose is to  address
 the problem of  optimal embeddings, in the whole of
$\rn$, for Sobolev spaces built upon  arbitrary
rearrangement-invariant spaces. Precisely, given any
rearrangement-invariant space $X(\rn)$, we find the smallest
rearrangement-invariant space $Y(\rn)$ which renders the embedding
\begin{equation}\label{XY}
W^m X(\rn) \to Y(\rn)
\end{equation}
true. Here,
 $W^mX(\rn)$ denotes the $m$-th order Sobolev type space built upon $X(\rn)$, and equipped with the norm
defined as in \eqref{sobnorm}, with $L^p(\rn)$ replaced with
$X(\rn)$.
  A rearrangement-invariant space is, in a sense, a
space of measurable functions endowed with a norm depending only on
the measure of the level sets of the functions. A precise definition
will be recalled in the next section.
\par Questions of this kind have been  investigated in the
literature  in the case when $\rn$ is replaced  with a domain of
finite measure.  The optimal target problem  has been solved in full
generality for these domains in \cite{CPS, T2}. Apart from the
examples mentioned above, few special instances are  known in the
entire $\rn$.  In this connection, let us mention that first-order
sharp Orlicz-Sobolev embeddings in $\rn$ are established in
\cite{Ci1}, and that the paper   \cite{V} contains the borderline
case $m=1$, $p=n$ of \eqref{lorentz}.
\par A key result in our approach is what we call a reduction
principle, asserting the equivalence between a Sobolev embedding of
the form \eqref{XY}, and a couple of  one-dimensional inequalities
involving  the representation norms of
 $X(\rn)$ and $Y(\rn)$ on $(0, \infty)$.
 This is the content of
Theorem \ref{T:reduction-higher}.
\par The optimal rearrangement-invariant target space $Y(\rn)$
for $W^m X(\rn)$ in \eqref{XY} is exhibited in Theorem
\ref{T:main-higher}.
 The conclusion
%identification of the optimal rearrangement-invariant target space in the Sobolev embedding \eqref{XY}
shows that the phenomenon  recalled above for the standard Sobolev
spaces $W^{m,p}(\rn)$ carries over to any space $W^mX(\rn)$: no
higher  integrability  of a function near  infinity follows from
membership of its derivatives in  $X(\rn)$,   whatever
rearrangement-invariant space is $X(\rn)$. Loosely speaking, the
norm in the optimal target space  behaves locally like  the optimal
 target norm for embeddings of the space $W^mX(\rn)$ with $\rn$ replaced by a bounded subset,   and like the norm of  $X(\rn)$  near infinity.
%Loosely speaking, such an optimal target equals
%the intersection of two spaces: the former is a local version of the
%optimal target for $W^mX(\rn)$ on subdomains of finite measure, the
%latter is a kind of restriction of $X(\rn)$ at infinity.
We stress that, since the norm in $X(\rn)$ is not necessarily of
integral type, a precise definition of the norm of the optimal
target space  is not straightforward, and the proofs of Theorems
\ref{T:main-higher} and \ref{T:reduction-higher}
 call for new ingredients compared to
the finite-measure framework.
%This is especially apparent in the
%arguments showing  the sharpness of our results.
\iffalse In this connection, let us mention that our approach relies
upon on a sharp iterative argument on the order $m$ of the Sobolev
space. Its implementation requires  a characterization, of
independent interest, of the optimal target space  for (first-order)
non-homogeneous Sobolev type spaces of functions. These are spaces
of weakly differentiable functions from some rearrangement-invariant
space,  whose derivatives belong  to a possibly different space from
the same class. \fi
\par
Theorems \ref{T:main-higher} and \ref{T:reduction-higher} can be
applied to derive optimal embeddings for customary and
unconventional spaces of Sobolev type. In particular, we are able
 to identify the optimal
Orlicz target and the optimal rearrangement-invariant target in
Orlicz-Sobolev embeddings (Theorems \ref{T:orlicz2} and
\ref{T:orlicz1}), and the optimal rearrangement-invariant target in
Lorentz-Sobolev embeddings (Theorem
\ref{T:application-to-lorentz-spaces}).

\section{Background}\label{S:preliminaries}

Let $\mathcal{E}\subset\Rn$ be a Lebesgue measurable set. We denote
by $\M(\mathcal{E})$ the set of all Lebesgue measurable functions
$u: \mathcal{E} \to [-\infty , \infty]$. Here, vertical bars
$|\cdot|$ stand for Lebesgue measure. We also define
$\M_+(\mathcal{E})= \{u \in \M(\mathcal{E}): u \geq 0\}$, and $\M
_0(\mathcal{E})= \{u \in \M(\mathcal{E}): u\ \textup{is finite a.e.
in}\ \mathcal{E}\}$. The \textit{non-increasing rearrangement}
$u\sp* : [0, \infty) \to [0,\infty]$ of a function $u \in
\M(\mathcal{E})$ is defined as
$$
u\sp*(s)=\inf\{t\geq 0:\,|\{x\in \mathcal{E} :\,|u(x)|>t\}|\leq s\}
\qquad\textup{for}\ s\in[0,\infty).
$$
We also define $u^{**} : (0,\infty) \to [0, \infty]$ as
$$
u^{**}(s) = \frac{1}{s}\int _0^s u^*(r)\, dr \qquad\textup{for}\
s\in(0,\infty).
$$
The operation of rearrangement is neither linear, nor sublinear.
However,
\begin{equation}
\label{sum*} (u+v)^*(s) \leq u^*(s/2) + v^*(s/2) \qquad \hbox{for
$s>0$,}
\end{equation}
for every $u, v \in \mathcal M_+(\mathcal{E})$.
%A property of
%rearrangements ensures that if $\{u_k\}$ is a sequence in $\M
%(\mathcal{E})$, and $u \in \M(\mathcal{E})$, then
%\begin{equation}\label{limit*}
%|u_k| \nearrow |u| \,\, \hbox{a.e. in $\mathcal{E}$ implies } \,\,
%u_k^* \nearrow u^*\,.
%\end{equation}

\par
Let $L\in(0,\infty]$. We say that a functional $\|\cdot\|_{X(0,L)}:
\M _+ (0, L) \to [0,\infty]$  is a \textit{function norm} if, for
every $f$, $g$ and $\{f_k\}$ in $\M_+(0, L)$, and every $\lambda
\geq0$, the following properties hold:
\begin{itemize}
\item[(P1)] $\|f\|_{X(0, L)}=0$ if and only if $f=0$ a.e.;
$\|\lambda f\|_{X(0, L)}= \lambda\|f\|_{X(0, L)}$; \par\noindent
\qquad $\|f+g\|_{X(0, L)}\leq \|f\|_{X(0, L)}+ \|g\|_{X(0, L)}$;
\item[(P2)] $f \le g$ a.e.\  implies $\|f\|_{X(0, L)} \le
\|g\|_{X(0, L)}$;
\item[(P3)] $f_k \nearrow f$ a.e.\ implies
$\|f_k\|_{X(0, L)} \nearrow \|f\|_{X(0, L)}$;
\item[(P4)] $\|\chi_{\mathcal F}\|_{X(0, L)}<\infty$ for every set
$\mathcal F\subset(0,L)$ of finite measure;
\item[(P5)] for every set $\mathcal F\subset(0,L)$ of finite measure
there exists a~positive constant $C_{\mathcal F}$ such that
$\int_{\mathcal F} f(s)\,ds \le C_{\mathcal F} \|f\|_{X(0, L)}$ for
every $f\in\M_+(0,L)$.
\end{itemize}
Here, and in what follows, $\chi_{\mathcal F}$ denotes the
characteristic function of a set $\mathcal F$.
\\
If, in addition,
\begin{itemize}
\item[(P6)]\qquad $\|f\|_{X(0, L)} = \|g\|_{X(0, L)}$ whenever
$f\sp* = g\sp *$,
\end{itemize}
we say that $\|\cdot\|_{X(0, L)}$ is a
\textit{rearrangement-invariant function norm}.
\par
The fundamental function  $\varphi_X : (0,L) \to [0, \infty)$ of a
rearrangement-invariant function norm $\|\cdot\|_{X(0,L)}$ is
defined as
\begin{equation}\label{fund}
\varphi_X(s)=\|\chi_{(0,s)}\|_{X(0,L)} \qquad \textup{for
$s\in(0,L)$.}
\end{equation}
If $\|\cdot\|_{X(0, L)}$ and $\|\cdot\|_{Y(0, L)}$ are
rearrangement-invariant function norms, then the functional
$\|\cdot\|_{(X\cap Y)(0, L)}$, defined by
\[
\|f\|_{(X\cap Y)(0, L)}=\|f\|_{X(0, L)} + \|f\|_{Y(0, L)}\,,
\]
is also a~rearrangement-invariant function norm.

With any rearrangement-invariant function norm $\|\cdot\|_{X(0, L)}$
is associated another functional on $\M_+(0, L)$, denoted by
$\|\cdot\|_{X'(0,L)}$, and defined as
\begin{equation*}%\label{X'}
\|g\|_{X'(0, L)}=\sup_{\begin{tiny} \begin{array}{c}
                       {f\in \M_+(0, L)}\\
                        {\|f\|_{X(0, L)}\leq1}
                      \end{array}
                      \end{tiny}}\int_0\sp L f(s)g(s)\,ds
\end{equation*}
for $g \in  \M_+(0, L)$. It turns out that  $\|\cdot\|_{X'(0,L)}$ is
also a~rearrangement-invariant function norm, which is called
the~\textit{associate function norm} of $\|\cdot\|_{X(0,L)}$. Also,
\begin{equation*}%\label{X''}
\|f\|_{X(0,L)}=\sup_ {\begin{tiny} \begin{array}{c}
                       {g\in \M_+(0, L)}\\
                        {\|g\|_{X'(0, L)}\leq1}
                      \end{array}
                      \end{tiny}}
\int_0\sp {L}f(s)g(s)\,ds =\|f\|_{X^{''}(0,L)}
\end{equation*}
for every $f\in\M_+(0,L)$.
\par
The generalized \textit{H\"older inequality}
\begin{equation*}%\label{hoelder}
\int_{0}^Lf(s) g(s)\,ds\leq\|f\|_{X(0,L)}\|g\|_{X'(0,L)}
\end{equation*}
holds for every $f \in \mathcal M_+(0,L)$ and $g \in \mathcal
M_+(0,L)$.

Given any $t>0$, the \textit{dilation operator} $E_{\lambda}$ is
defined at $f\in \M(0,L)$ by
\begin{equation*}%\label{dilation}
  (E_{\lambda}f)(s)=\begin{cases}
  f(s/\lambda)\quad&\textup{if}\;\; s\in (0, \lambda L)\\
  0&\textup{if}\;\; s\in [\lambda L, L)\,.
  \end{cases}
\end{equation*}
 The inequality
$$
\|E_\lambda f\|_{X(0,L)} \leq \max\{1, 1/\lambda\} \|f\|_{X(0,L)}
$$
holds for every rearrangement-invariant function norm
$\|\cdot\|_{X(0,L)}$, and for every $f \in \M_+(0,L)$. \iffalse

\\
 DEFINE ON $(0,L)$ \quad For each fixed $t\in(0,\infty)$,
the dilation operator $E_t$, defined by the formula $E_tf(s)=f(ts)$
for $f\in\M_+(0,\infty)$ and $t\in(0,\infty)$, satisfies the
inequality
$$
\|E_t f\|_{X(0,\infty)} \leq C \|f\|_{X(0,\infty)}
$$
for every rearrangement-invariant function norm
$\|\cdot\|_{X(0,\infty)}$.

\fi
\par
A ``localized'' notion of a rearrangement-invariant function norm
 will be needed for our purposes. The
localized rearrangement-invariant function norm of
$\|\cdot\|_{X(0,\infty)}$ is denoted by $\|\cdot\|_{\widetilde
X(0,1)}$ and defined as follows.  Given a function $f\in\M_+(0,1)$,
we call  $\widetilde f$ its continuation to $(0, \infty)$ by $0$
outside $(0,1)$, namely
\[
\widetilde f(t)=
\begin{cases}
f(s) &\textup{if}\;\;s\in(0,1),\\
0 &\textup{if}\;\; s\in[1,\infty).
\end{cases}
\]
Then, we define the functional $\|\cdot\|_{\widetilde X(0,1)}$ by
\begin{equation}\label{tilde}
\|f\|_{\widetilde X(0,1)}=\|\widetilde f\|_{X(0,\infty)}
\end{equation}
for $f\in\M_+(0,1)$. One can  verify that $\|\cdot\|_{\widetilde
X(0,1)}$ is actually a~rearrangement-invariant function norm.
\par
Given a measurable set $\mathcal{E}\subset \rn$ and a
rearrangement-invariant function norm
$\|\cdot\|_{X(0,|\mathcal{E}|)}$, the space $X(\mathcal{E})$ is
defined as the collection of all  functions  $u \in\M(\mathcal{E})$
such that the quantity
\[
\|u\|_{X(\mathcal{E})}=\|u\sp* \|_{X(0,|\mathcal{E}|)}
\]
is finite. The functional   $\|\cdot\|_{X(\mathcal{E})}$ defines a
norm on $X(\mathcal{E})$, and the latter is a Banach space  endowed
with this norm, called a~\textit{rearrangement-invariant space}.
Moreover, $X(\mathcal{E}) \subset \M_0(\mathcal{E})$ for any
rearrangement-invariant space $X(\mathcal{E})$.
\par
The rearrangement-invariant space $X'(\mathcal{E})$ built upon the
function norm $\|\cdot \|_{X'(0,|\mathcal{E}|)}$ is called the
\textit{associate space} of $X(\mathcal{E})$.
\par
Given any rearrangement-invariant spaces $X(\mathcal{E})$ and
$Y(\mathcal{E})$, one has that
\begin{equation}\label{E:equivemb}
 X(\mathcal{E}) \rightarrow Y(\mathcal{E}) \qquad \hbox{if and only
 if} \qquad
Y'(\mathcal{E}) \rightarrow X'(\mathcal{E}),
\end{equation}
with the same embedding norms.

We refer the reader to~\cite{BS} for a comprehensive treatment of
rearrangement-invariant spaces. \iffalse

Our principal interest in this paper concerns the case $E=\rn$,
where $n\in\N$, $n\geq 2$. We shall however also work with the cases
when $n=1$ end either $E=(0,\infty)$ or $E=(0,1)$. Thus,
accordingly, we will usually have either $L=1$ or $L=\infty$.

\fi
\par
In remaining part of this section, we recall the definition of some
 rearrangement-invariant function norms and spaces that, besides the Lebesgue
spaces, will be called into play in our discussion.
\\ Let us begin with the Orlicz spaces, whose definition makes use of the notion of  Young function. A Young function $A: [0, \infty ) \to [0,
\infty ]$ is a convex (non trivial), left-continuous function
vanishing at $0$. Any such function takes the form
\begin{equation}\label{young}
A(t) = \int _0^t a(\tau ) d\tau \qquad  \hbox{for $t \geq 0$},
\end{equation}
for some non-decreasing, left-continuous function $a: [0, \infty )
\to [0, \infty ]$ which is neither identically equal to $0$, nor to
$\infty$.  The Orlicz space built upon the Young function $A$  is
associated
 with  the
 \emph{Luxemburg function norm} defined  as
\begin{equation*}%\label{lux}
\|f\|_{L^A(0,L)}= \inf \left\{ \lambda >0 :  \int_0\sp L A \left(
\frac{f(s)}{\lambda} \right) ds \leq 1 \right\}
\end{equation*}
for $f \in \M _+(0,L)$. In particular, given a measurable set
$\mathcal{E} \subset \mathbb R^n$, one has that $L^A (\mathcal{E})=
L^p (\mathcal{E})$ if $A(t)= t^p$ for some $p \in [1, \infty )$, and
$L^A (\mathcal{E})= L^\infty (\mathcal{E})$ if $A(t)= \infty
\chi_{(1, \infty)}(t)$.

A Young function $A$ is said to dominate another Young function $B$
near infinity [resp. near zero] [resp. globally] if there exist
positive constants $c>0$ and $t_0>0$  such that
\[
B(t)\leq A(c\, t) \qquad \textrm{for  $t\in  (t_0, \infty)$ \, \,
[$t\in  [0, t_0)$] \,\, [$t \in [0, \infty)$]\,.}
\]
The functions $A$ and $B$ are called equivalent near infinity [near
zero] [globally] if they dominate each other near infinity [near
zero] [globally]. Equivalence of the Young functions $A$ and $B$
will be denoted by $A \approx B$.
\\
If $|\mathcal{E}|= \infty$, then
\begin{equation*}%\label{B.6}
L^A(\mathcal{E})\to L^B(\mathcal{E}) \qquad \textup{if and only
if}\qquad
 \hbox{$A$ dominates $B$ globally}\,.
\end{equation*}
If $|\mathcal{E}|< \infty$, then
\begin{equation*}%\label{B.6bis}
L^A(\mathcal{E})\to L^B(\mathcal{E}) \qquad \textup{if and only
if}\qquad
 \hbox{$A$ dominates $B$ near infinity}\,.
\end{equation*}

\iffalse
 We denote by $L^p\log^\alpha L(\Omega , \nu)$ the Orlicz space
associated with a Young function equivalent to $t^p (\log t)^\alpha$
near infinity, where either $p>1$ and $\alpha \in \R$, or $p=1$ and
$ \alpha \geq 0$. The notation $\exp L^\beta (\Omega ,\nu)$ will be
used for the Orlicz space built upon a Young function equivalent to
$e^{t^\beta}$ near infinity, where $\beta>0$. Also, $\exp\exp
L^\beta (\Omega ,\nu)$ stands for the Orlicz space associated with a
Young function equivalent to $e^{e^{t^\beta }}$ near infinity. \fi

Given $1 \leq p,q\leq\infty$, we define the Lorentz functional
$\|\cdot\|_{L\sp{p,q}(0,L)}$ by
\[
\|f\|_{L\sp{p,q}(0,
L)}=\Big\|f\sp*(s)s\sp{\frac1p-\frac1q}\Big\|_{L\sp{q}(0,L)}\,
\]
for $f \in\M_+(0,{L})$. This functional is a rearrangement-invariant
function norm provided that $1 \leq q \leq p$, namely when $q \geq
1$ and the weight function $s\sp{\frac1p-\frac1q}$ is
non-increasing. In general, it is known to be equivalent, up to
multiplicative constants, to a rearrangement-invariant function norm
if (and only if) either $p=q=1$, or $1<p<\infty$ and $1\leq
q\leq\infty$, or $p=q=\infty$. The rearrangement-invariant space  on
a measurable set $\mathcal{E} \subset \rn$, built upon the latter
rearrangement-invariant function norm, is the standard Lorentz space
$L^{p,q}(\mathcal{E})$. Thus,  $L^{p,q}(\mathcal{E})$ consists of
all functions $u \in \M(\mathcal{E})$ such that the functional
$\|u\|_{L^{p,q}(\mathcal{E})}$, defined as
\begin{equation}\label{may20}
\|u\|_{L\sp{p,q}(\mathcal{E})}=\|u^*\|_{L\sp{p,q}(0, |\mathcal{E}|)}
\end{equation}
 is finite.
\iffalse

It is known (see e.g.~\cite[Chapter~8]{FS}) that the functional
$\|\cdot\|_{L\sp{p,q}(0, L)}$ is equivalent to a
rearrangement-invariant function norm, up to multiplicative positive
constants,  if and only if $p=q=1$ or $1<p<\infty$ and $1\leq
q\leq\infty$ or $p=q=\infty$. This follows from the fact that
replacing   $f\sp{*}$ with $f\sp{**}$ in the definition of
$\|\cdot\|_{L\sp{p,q}(0,L)}$ results in an equivalent functional in
the case when $p>1$, whereas it is obvious when $p=q=1$, since
$L\sp{1,1}(0,L)=L\sp1(0,L)$.

We will say that, given $\mathcal{E}\subset\Rn$, a set
$Y\subset\M(\mathcal{E})$, endowed with a functional
$u\mapsto\|u\|_{Y}$, is \textit{equivalent to a
rearrangement-invariant space}, if there exists a
rearrangement-invariant space $X(\mathcal{E})$ such that
$Y=X(\mathcal{E})$ in the set-theoretical sense and there exists
a~positive constant $C$ such that, for every $u\in\M(\mathcal{E})$,
\[
C\sp{-1}\|u\|_{Y}\leq \|u\|_{X(\mathcal{E})} \leq C\|u\|_{Y}.
\]
Likewise, we will say that a functional $F\colon\M(\mathbb R\sp
n)\to[0,\infty)$ is \textit{equivalent to a rearrangement-invariant
norm} if there exists a rearrangement-invariant norm $\|\cdot\|$ and
a  constant $C$ such that, for every $u\in\M(\R\sp n)$, one has
\[
C\sp{-1}\|u\|\leq F(u)\leq C\|u\|.
\]

\fi
\\
The notion of Lorentz functional can be generalized on replacing the
function $s\sp{\frac1p-\frac1q}$ by a more general weight $w(s) \in
\M_+(0,L)$ in \eqref{may20}. The resulting functional is classically
denoted by $\|\cdot\|_{\Lambda\sp{q}(w)(0,L)}$, and reads
$$\|f\|_{\Lambda\sp{q}(w)(0,L)}=\|f\sp*(t)w(t)\|_{L\sp q(0,L)}$$
for $f \in\M_+(0,{L})$. A characterization of those exponents $q$
and  weights $w$ for which the latter functional is equivalent to a
rearrangement-invariant function norm is known. In particular, when
$q \in (1, \infty)$, this is the case if and only if there exists a
positive constant $C$ such that
\[
s\sp{q} \int_{s}\sp{L} \frac{w(r)\sp q}{r\sp q}\,dr \leq C \int_0\sp
s w(r)\sp q\,dr \quad \hbox{for $0< s < L$}
\]
\cite[Theorem~2]{Sa}. Moreover,
 $\|\cdot\|_{\Lambda\sp{1}(w)(0,L)}$ is
equivalent to a rearrangement-invariant function norm   if and only
if there exists a  constant $C$ such that
\[
\frac{1}{s}\int_{0}\sp{s}w(t)\,dt \leq
\frac{C}{r}\int_{0}\sp{r}w(t)\,dt \quad \hbox{for $0<r\leq s < L$}
\]
 \cite[Theorem~2.3]{CGS}.  Finally,
$\|\cdot\|_{\Lambda\sp{\infty}(w)(0,L)}$ is equivalent to a
rearrangement-invariant function norm   if and only if there exists
a constant $C$ such that
\[
\int_{0}\sp{s} \frac{dr}{w(r)}\leq \frac {Cs}{w(s)} \quad \hbox{for
$0< s < L$.}
\]
This follows from a result of ~\cite[Theorem~3.1]{So}. \iffalse for
a weight $w$ of the form $w(t)=\int_0\sp t\rho(s)\,ds$ and can be
obtained for a general $w$ by the following approximation argument.
We first note that
\[
\sup_{0<s<\infty}w(s)f\sp*(s) =
\sup_{0<s<\infty}\overline{w}(s)f\sp*(s),
\]
where $\overline{w}(s)=\sup_{0<y\leq s}w(y)$ for $s\in(0,\infty)$.
Indeed, this follows from
\begin{align*}
\sup_{0<s<\infty}\overline{w}(s)f\sp*(s) &= \sup_{0<s<\infty}
\sup_{0<y\leq s}w(y)f\sp*(s) = \sup_{0<y<\infty}w(y)
\sup_{y\leq s<\infty}f\sp*(s)\\
&= \sup_{0<y<\infty}\overline{w}(y)f\sp*(y).
\end{align*}
Consequently, we can always with no loss of generality assume that
the weight $w$ in $\Lambda\sp{\infty}(w)$ is non-decreasing.
Finally, we approximate a non-decreasing function $w$ by a sequence
of functions of the above form and use the Monotone Convergence
Theorem.

\fi In any of these cases, we shall denote by
 $\Lambda\sp{q}(w)(\mathcal{E})$ the corresponding rearrangement-invariant space on a measurable set $\mathcal{E}\subset \rn$.
\\
A further extension of the notion of the Lorentz functional
$\|\cdot\|_{\Lambda\sp{q}(w)(0,L)}$ is obtained when the role of the
function norm $\|\cdot \|_{L^q(0,L)}$ is played by a more general
Orlicz function norm $\|\cdot \|_{L^A(0,L)}$. The  resulting
functional will be denoted by $\|\cdot\|_{\Lambda\sp{A}(w)(0,L)}$,
and defined as
\[
\|f\|_{\Lambda\sp{A}(w)(0,L)}=\|f\sp*(s)w(s)\|_{L\sp{A}(0,L)}
\]
for $f \in \M_+(0,L)$. A complete characterization of those Young
functions $A$ and weights $w$ for which the functional
$\|\cdot\|_{\Lambda\sp{A}(w)(0,L)}$ is equivalent to a
rearrangement-invariant function norm seems not to be available in
the literature. However, this functional is certainly equivalent to
a rearrangement-invariant function norm whenever the weight $w$ is
equivalent, up to multiplicative constants, to a non-increasing
function. This is the only case that will be needed in our
applications. For such a choice of weights $w$, we shall denote by
$\Lambda\sp{A}(w)(\mathcal{E})$ the corresponding
rearrangement-invariant space on a measurable set
$\mathcal{E}\subset \rn$. \iffalse Let $A$ be a Young function and
let $w$ be a \textit{weight} (a~positive measurable function) on
$(0,\infty)$. Then we define the functional
$\|\cdot\|_{\Lambda\sp{A}(w)(\rn)}$ on $\M(\rn)$ by
\[
\|u\|_{\Lambda\sp{A}(w)(\rn)}=\|u\sp*(t)w(t)\|_{L\sp{A}(0,\infty)}
\]
and we denote by $\Lambda\sp{A}(w)(\rn)$ the collection of all
functions from $u\in\M(\mathbb R\sp n)$ for which
$\|u\|_{\Lambda\sp{A}(w)(\rn)}$ is finite. We note that
$\|\cdot\|_{\Lambda\sp{A}(w)(\rn)}$ is a rearrangement-invariant
norm whenever $w$ is non-increasing on $(0,\infty)$. \fi
\\
A comprehensive treatment of Lorentz type functionals can be found
in~\cite{CPSS} and~\cite{FS}.

\section{Main results}\label{S:main}

Our first main result exhibits the optimal rearrangement-invariant
space into which any assigned  arbitrary-order Sobolev space is
embedded. Given an open set $\Omega$ in $\rn$ and any
rearrangement-invariant space $X(\Omega)$, the $m$-th order Sobolev
type space $W^m X(\Omega)$ is defined as
% \begin{align*}W^m(X)(\Omega) = \{  u\in X(\Omega): u \;& \hbox{is
% $m$-times weakly differentiable,} \\
% & \hbox{and $|\nabla ^k u|\in X(\Omega)$ for $k=1, \dots, m$}\}
% \end{align*}
\begin{align*}W^m(X)(\Omega) = \{  u\in X(\Omega): u \;\hbox{is
$m$-times weakly differentiable, and $|\nabla ^k u|\in X(\Omega)$
for $k=1, \dots, m$}\}.
\end{align*}
The space $W^m X(\Omega)$ is a Banach space endowed with the norm
given by
\begin{equation}\label{WmX}
\|u\|_{W^m X(\Omega)} = \sum _{k=0}^{m} \|\nabla ^k u\|_{X(\Omega)}
\end{equation}
for $u \in \Omega$.
 In \eqref{WmX}, and in what follows, the
rearrangement-invariant norm of a vector is understood as the norm
of its length.

Given a rearrangement-invariant space $X(\rn)$, we say that $Y(\rn)$
is the \textit{optimal rearrangement-invariant target space} in
the~Sobolev embedding
\begin{equation}\label{embrn}
W\sp{m}X(\rn)\rightarrow Y(\rn)
\end{equation}
 if it is the smallest one that renders \eqref{embrn} true, namely, if for every rearrangement-invariant space
$S(\rn)$ such that $W\sp{m}X(\rn)\rightarrow S(\rn)$, one has that
$Y(\rn)\rightarrow S(\rn)$. The construction of the optimal space
$Y(\rn)$ in \eqref{embrn} requires a few steps.
 A first ingredient is the function norm $\|\cdot \|_{Z(0,1)}$  obeying
\begin{align}\label{Z}
\|g\|_{Z' (0,1)}=\|s\sp{\frac mn}g\sp{**}(s)\|_{\widetilde
X'(0,1)}\qquad \textup{for}\ \quad g\in\M(0,1),
\end{align}
where the localized  rearrangement-invariant function norm $\|\cdot
\|_{\widetilde{X}(0,1)}$ is defined as in \eqref{tilde}, and
$\|\cdot \|_{\widetilde{X}'(0,1)}$ stands for its associated
function norm. The function norm $\|\cdot \|_{Z (0,1)}$ determines
the optimal rearrangement-invariant space in the Sobolev embedding
for the space $W^m\widetilde X(\Omega)$ in any regular open set in
$\rn$ (with $|\Omega |=1$),   see \cite{CPS, T2}.
%
%We begin with optimal rearrangement-invariant
%target space in the Sobolev embedding for a ``localized'' version of
%the space $X(\rn)$. The norm in this space is determined by the
%function norm $\|\cdot\|_{Z(0,1)}$  obeying
%\begin{equation}\label{Z}
%\|g\|_{Z'(0,1)}=\|t\sp{\frac mn}g\sp{**}(t)\|_{\widetilde
%X'(0,1)}\qquad \textup{for}\ \quad g\in\M(0,1),
%\end{equation}
%where the rearrangement-invariant function norm $\tilde{X}$ is
%defined as in \eqref{tilde}, and $\widetilde X'(0,1)$ stands for
%$(\widetilde X)'(0,1)$ (see \cite{CPS, T2}).
\\
Next, we extend $\| \cdot\|_{Z(0,1)}$ to a function norm  $\|
\cdot\|_{X^m(0,\infty)}$  on $(0, \infty)$ by setting
\begin{equation}\label{E:definition-of-Y-opt-2-higher}
\|f\|_{X\sp m(0,\infty)} = \|f\sp*\|_{Z(0,1)}\qquad \textup{for}\
f\in\M(0,\infty).
\end{equation}
The optimal rearrangement-invariant target space $Y(\rn)$ in
embedding \eqref{embrn} is the space $X_{\opt}^m(\rn) $ defined as
\begin{equation}\label{optspace}
X_{\opt}^m(\rn)  = (X^m \cap X)(\rn).
\end{equation}
 One has that
 $$\{u \in X^m(\rn): |{\rm supp}\; u|< \infty\} \subset X(\rn).$$
 This follows via  Proposition \ref{ZX}, Section
 \ref{S:proofs}.
 %On the other hand, by equations \eqref{Z}
  %and \eqref{E:definition-of-Y-opt-2-higher}, the norm
  %in $X^m(\rn)$ only depends on the restriction of functions
  %to subsets of finite measure.
  Hence, if $|{\rm supp}\; u|< \infty$, then $u \in X_{\opt}^m(\rn)$ if and only if  $u \in X^m (\rn)$.
    In this sense the optimal space $X_{\opt}^m(\rn) $ is determined  by $X^m(\rn)$ locally, and  by $X(\rn)$ near infinity.
  \iffalse

  Thus, loosely speaking,
  equation \eqref{optspace} tells us that the optimal target
  space for $W^mX(\rn)$ is determined by the local optimal
  target space for $X(\rn)$ on sets of finite measure, and
  by $X(\rn)$ itself near infinity.
\fi

\begin{theorem}[\textbf{Optimal Sobolev embedding}]\label{T:main-higher}
Let $n \geq 2$, $m \in \N$, and let $X(\rn)$  be
a~rearrangement-invariant space.   Then
\begin{equation}\label{E:sobolev-opt-higher}
W\sp{m}X(\rn)\to X_{\opt}^m(\rn) \,,
\end{equation}
where $X_{\opt}^m(\rn) $ is the space defined by \eqref{optspace}.
 Moreover,  $X_{\opt}^m(\rn) $ is the optimal rearrangement-invariant target space
  in ~\eqref{E:sobolev-opt-higher}.
\end{theorem}

\iffalse Note that the function norm $\| \cdot \|_{Z(0,1)}$, defined
by \eqref{Z}, is equivalent to   $\| \cdot \|_{L^\infty (0, 1)}$, up
to multiplicative constants, provided that
\begin{equation}\label{conv}
\| \chi _{(0,1)}(t) t^{-1+\frac mn} \|_{X'(0,\infty)}< \infty.
\end{equation}
Indeed, if~\eqref{conv} holds, then we have, for every
$g\in\M(0,1)$,
\[
\|g\|_{Z'(0,1)}=\|t\sp{-1+\frac mn}\int_0\sp
tg\sp*(s)\,ds\|_{\widetilde X'(0,1)} \leq\int_0\sp tg\sp*(s)\,ds\|
\chi _{(0,1)}(t) t^{-1+\frac mn} \|_{X'(0,\infty)} \leq
C\|g\|_{L\sp1(0,1)},
\]
establishing the embedding $L\sp1(0,1)\to Z'(0,1)$, while the
converse embedding follows from (P5). Thus $Z'(0,1)=L\sp{1}(0,1)$,
in other words, $Z(0,1)=L\sp{\infty}(0,1)$. Thus,

\fi Theorem \ref{T:main-higher} has the following  consequence.

\begin{corollary}[\textbf{Supercritical Sobolev embedding}]\label{T:supercrit}
Let $n \geq 2$, $m \in \N$, and let $X(\rn)$  be
a~rearrangement-invariant space such that
\begin{equation}\label{conv}
\| \chi _{(0,1)}(s) s^{-1+\frac mn} \|_{X'(0,\infty)}< \infty.
\end{equation}
\textup{(}In particular, this is always the case when $m \geq n$,
whatever is $X(\rn)$.\textup{)}
 Then
  $X_{\opt}^m(\rn)=(L^\infty\cap X)(\rn)$, up to equivalent norms. Namely,
  \begin{equation}\label{E:sobolev-opt-super}
W\sp{m}X(\rn)\to (L^\infty  \cap X)(\rn) \,,
\end{equation}
  and $(L^\infty  \cap X)(\rn)$ is the optimal rearrangement-invariant target space in \eqref{E:sobolev-opt-super}.
\end{corollary}

Our second main result is  a~reduction theorem for Sobolev
embeddings on the entire $\rn$.

\begin{theorem}[\textbf{Reduction principle}]\label{T:reduction-higher}
Let $X(\rn)$  and $Y(\rn)$ be rearrangement-invariant spaces. Then
\begin{equation}\label{E:reduction1}
W\sp{m}X(\rn)\hra Y(\rn)
\end{equation}
if and only if there exists a constant $C$ such that
\begin{equation}\label{E:reduction2}
\bigg\|\chi _{(0,1)}(s)\int _s^1 f(r) r^{-1+\frac mn}\, dr\bigg\|_{Y
(0,\infty)} \leq C \|\chi _{(0,1)} f\|_{ X (0,\infty)}
\end{equation}
and
\begin{equation}\label{E:reduction3}
\|\chi _{(1,\infty)} f\|_{Y(0, \infty)} \leq C \|f\|_{X (0,\infty)}
\end{equation}
for every non-increasing function $f: (0, \infty ) \to [0, \infty)$.
\end{theorem}

\begin{remark}\label{fdecreasing}
{\rm In fact, inequality \eqref{E:reduction2} turns out to hold for
every non-increasing nonnegative function $f$ if and only if it
holds for every nonnegative function $f$ -- see \cite[Corollary
9.8]{CPS}}.
\end{remark}

In the remaining part of this section we present applications of our
general results to two important families of Sobolev type  spaces:
the Orlicz-Sobolev and the Lorentz-Sobolev spaces. In the light of
Corollary \ref{T:supercrit}, we shall mainly restrict our attention
to the case when $m< n$.

\iffalse We may assume, without loss of generality, that
\begin{equation}\label{orlicztr1}
\int _0 \bigg(\frac t{A(t)}\bigg)^{\frac {m}{n-m}} dt < \infty.
\end{equation}
Indeed, $A$ can be replaced, if necessary, by a Young function
equivalent near infinity, which renders \eqref{orlicztr1} true, such
replacement leaving $W^{m}L^A(\Omega )$ unchanged (up to equivalent
norms).
\par\noindent
\fi

Let us begin with   embeddings for Orlicz-Sobolev spaces built upon
a Young function $A$. The  notation $W^{m,A}(\rn)$ will also be
adopted instead of $W^mL^A(\rn)$. This kind of spaces is of  use in
applications to the theory of partial differential equations, whose
nonlinearity is not necessarily of power type,  which also arise in
mathematical models for diverse physical phenomena -- see e.g.
\cite{AcerbiMingione, Ball, Baroni, BreitSchirra, BreitStrofVerde,
BulDiSchw, BulMaMa, Eyring, Korolev, Lieberman, Marcellini,
Talenti79, Talenti90, Wr}.
%and
%\begin{equation}\label{orlicztr1bis}
%\int ^\infty \bigg(\frac t{A(t)}\bigg)^{\frac {m}{n-m}} dt = \infty\,,
%\end{equation}
\\ We first focus on embeddings of $W^mL^A(\rn)$ into  Orlicz spaces. We  show that an optimal target space in this class always exists, and we exhibit it explicitly.
\\
Define the function $H_{\frac nm} : [0, \infty ) \to [0, \infty )$
as
\begin{equation}\label{orlicztr2}
H_{\frac nm}(\tau) = \bigg(\int _0^\tau \bigg(\frac t
{A(t)}\bigg)^{\frac {m}{n-m}} dt\bigg)^{\frac {n-m}n} \qquad
\hbox{for $\tau\geq 0$,}
\end{equation}
and the Young function $A_{\frac nm}$ as
\begin{equation}\label{Anm}
A_{\frac nm}(t) = A \big(H_{\frac nm}^{-1} (t)\big) \qquad \hbox{for
$t \geq 0$,}
\end{equation}
where $H_{\frac nm}^{-1}$ denotes the (generalized) left-continuous
inverse of $H_{\frac nm}$ (see \cite{Ci4, Ci3}).

\begin{theorem}[\textbf{Optimal Orlicz target in Orlicz-Sobolev embeddings}]\label{T:orlicz2} Let $n \geq 2$, $m \in \N$, and
let $A$ be a Young function. Assume that   $m<n$,  and let $A_{\frac
nm}$ be the Young function defined by \eqref{Anm}. Let $A_{\rm opt}$
be a Young function such that
\begin{equation*}
A_{\rm opt}(t)\approx \begin{cases} A(t) &\textup{near}\ 0
\\
A_{\frac nm}(t) &\textup{near infinity}\,.
\end{cases}
\end{equation*}
Then
\begin{equation}\label{33}
W^{m, A}(\rn) \to L^{A_{\rm opt}}(\rn) ,
\end{equation}
and $L^{A_{\rm opt}}(\rn)$ is the optimal Orlicz target space in
\eqref{33}.
\end{theorem}

 \begin{remark}\label{replace}
In equations \eqref{orlicztr2} and \eqref{Anm}, we may assume,
without loss of generality, that
\begin{equation}\label{orlicztr1}
\int _0 \bigg(\frac t{A(t)}\bigg)^{\frac {m}{n-m}} dt < \infty\,,
\end{equation}
so that $H$ is well defined. Indeed, since only the behavior of
$A_{\frac nm}$, and hence of $H$, near infinity is relevant in view
of our result, $A$ can be replaced in \eqref{orlicztr2} and
\eqref{Anm} (if necessary) by a Young function equivalent near
infinity, which renders \eqref{orlicztr1} true.  Any such a
replacement results into an equivalent  function $A_{\frac nm}$ near
infinity.
\end{remark}

\begin{remark}\label{infinity} Notice that, if
\begin{equation}\label{orlicztr1bis}
\int ^\infty \bigg(\frac t{A(t)}\bigg)^{\frac {m}{n-m}} dt <
\infty\,,
\end{equation}
then $H^{-1}(t)= \infty$ for large $t$, and equation \eqref{Anm} has
accordingly to be interpreted as $A_{\frac nm}(t)= \infty$ for large
$t$.
\end{remark}

%%%%%%%%%%%%%%%%%%%%%%%%%%%%%%%%%%% Example 1 %%%%%%%%%%%%%%%%%%%%%%%%%%%%%%%

\begin{remark}\label{Orliczsuper}
Condition \eqref{orlicztr1bis} can be shown to agree with
\eqref{conv} when $X(\rn) = L^A(\rn)$, and    embedding \eqref{33}
recovers the conclusion of Corollary \ref{T:supercrit} in this case.
Indeed, by Remark \ref{infinity}, under assumption \eqref{conv} one
has that
\begin{equation*}
A_{\rm opt} (t) \approx \begin{cases} A(t) &\textup{if $\;\;
t\in[0,1]$}
\\ \infty &\textup{if $\;\; t\in
(1, \infty)$}\,.
\end{cases}
\end{equation*}
Owing to Lemma \ref{lemma2} below, the resulting Orlicz target space
$L^{A_{\rm opt}}(\rn)$ coincides with $(L^\infty \cap L^A)(\rn)$.
Hence, the target space $L^{A_{\rm opt}}(\rn)$ is in fact also
optimal  among all rearrangement-invariant spaces in  \eqref{33}.
\end{remark}

We next determine the optimal target space  for embeddings of
$W^{m,A}(\rn)$ among all rearrangement-invariant spaces. Such a
space turns out to be a Lorentz-Orlicz space defined as follows.
\\
Let $A$ be a Young function.
%of the form
%$$A(t) = \int_0^t a(s) \; ds \qquad \hbox{for} \ t\geq 0.
%$$
In the light of Remark \ref{Orliczsuper}, we may restrict our
attention to the case when $m<n$, and
\begin{equation}\label{divinf}
\int ^\infty \bigg(\frac t{A(t)}\bigg)^{\frac m{n-m}}\, dt =
\infty\,.
\end{equation}
Define the Young function $\widehat A$  by
\begin{equation}\label{B}
\widehat A(t)= \int_0^t \widehat  a(\tau) \; d\tau \qquad \hbox{for}
\ t\geq 0,
\end{equation}
where
\begin{equation*}%\label{1.7}
\widehat  a ^{\,-1}(s)\,=\, \left(\int_{a^{-1}(s)}^{\infty}
\left(\int_{0}^{t}\left(\frac{1}{a(\tau)}\right)^{\frac{m}{n-m}}d\tau\right)^{-\frac
nm} \frac{dt}{a(t)^{\frac n{n-m}}}\right)^{\frac{m}{m-n}} \qquad
\textrm{for} \,\, s\geq 0\,,
\end{equation*}
and $a$ is the function appearing in
 \eqref{young}.
Notice that condition \eqref{orlicztr1} is equivalent to requiring
that
\begin{equation}\label{conva}
\int _0 \bigg(\frac 1{a(\tau)}\bigg)^{\frac {m}{n-m}} d\tau <\infty
\,.
\end{equation}
 Thus, by   a reason
analogous to that explained in Remark \ref{replace}, there is no
loss of generality in assuming \eqref{conva}.

\begin{theorem}[\textbf{Optimal rearrangement-invariant target in Orlicz-Sobolev embeddings}]\label{T:orlicz1}
Let $n \geq 2$ and $m \in \N$,  and let $A$ be a Young function.
Assume that  $m<n$, and condition \eqref{divinf} is fulfilled, and
let $\widehat{A}$ be the Young function given by \eqref{B}. Let $E$
be a Young function such that
\begin{equation*}
E(t)\approx \begin{cases} A(t) &\textup{near}\ 0
\\  \widehat{A}(t) &\textup{near infinity}\,.
\end{cases}
\end{equation*}
and let $v: (0, \infty)\to (0, \infty)$ be the function defined as
\begin{equation}\label{v}
v(s)= \begin{cases} s^{-\frac mn} &\textup{if}\;\; s\in (0, 1)
\\
1 &\textup{if}\;\; s\in [1, \infty).
\end{cases}
\end{equation}
Then
\begin{equation}\label{10bis}
W^{m, A}(\rn) \to  \Lambda ^E(v)(\rn),
\end{equation}
and the Orlicz-Lorentz space $\Lambda ^E(v)(\rn)$ is the optimal
rearrangement-invariant target space in \eqref{10bis}.

\iffalse

\par\noindent (i) Assume that   condition \eqref{divinf} is fulfilled.
Let $E$ be a Young function such that
\begin{equation*}
E(t)\approx \begin{cases} B(t) &\textup{near infinity}
\\
A(t) &\textup{near}\ 0 ,
\end{cases}
\end{equation*}
and let $v: (0, \infty)\to (0, \infty)$ be the function defined as
\begin{equation}\label{v}
v(s)= \begin{cases} s^{-\frac mn} &\textup{if}\;\; s\in (0, 1)
\\
1 &\textup{if}\;\; s\in [1, \infty).
\end{cases}
\end{equation}
Then
\begin{equation}\label{10bis}
W^{m, A}(\rn) \to  \Lambda ^E(v)(\rn),
\end{equation}
and $\Lambda ^E(v)(\rn)$ is the optimal rearrangement-invariant
target space in \eqref{10bis}.
\par\noindent
(ii) Assume that condition \eqref{orlicztr1bis} is fulfilled.
%\begin{equation}\label{convinf}
%\int ^\infty \bigg(\frac t{A(t)}\bigg)^{\frac 1{n-1}}\, dt< \infty.
%\end{equation}
Let $A^\infty$ be the Young function defined as
\begin{equation*}
A^\infty (t) = \begin{cases} \infty &\textup{if $\;\;t\in (1,
\infty)$}
\\
A(t) &\textup{if $\;\;t\in[0,1]$}.
\end{cases}
\end{equation*}
Then
\begin{equation}\label{10ter}
W^{m, A}(\rn) \to  L^{A^\infty}(\rn),
\end{equation}
and the Orlicz space $L^{A^\infty}(\rn)$ is the optimal
rearrangement-invariant target space in \eqref{10ter}. \fi
\end{theorem}

\iffalse Observe that, by Remark \eqref{infinity}, embedding
\eqref{10ter} agrees with embedding \eqref{33} when condition
\eqref{orlicztr1bis} is in force. Thus, the optimal Orlicz target
space is also the optimal rearrangement-invariant target space in
this case. \fi

\begin{example}\label {example1}  We focus here on Orlicz-Sobolev spaces built upon a special family of Orlicz spaces, called Zygmund spaces and denoted by $L^p(\log L)^\alpha(\rn)$, where either $p = 1$ and $\alpha\geq 0$, or $p > 1$ and $\alpha\in \R$. They are associated with a Young function equivalent to $t^p(1+ \log_+ t)^\alpha$, where  $\log_+ (t) = \max \{\log t, 0\}$. Of course, $L^p(\log L)^0(\mathbb R\sp n) = L^p(\rn)$.
\\ Owing to
Theorem \ref{T:orlicz2} and \cite[Example 3.4]{Ci3}, one has that
$$W^m L^p(\log L)^\alpha (\rn) \to L^G(\rn)\,,$$
where
%
%
%
%
%\begin{equation*}
%A_{\rm opt}(t) \approx \begin{cases} A_{n}(t) &\textup{near
%infinity}
%\\
%t^p\log ^\alpha(1+t) &\textup{near}\ 0\, ,
%\end{cases}
%\end{equation*}
%where
\begin{equation*}
G(t) \approx \begin{cases} t^{\frac{np}{n-mp}}(\log t)^{\frac
{n\alpha}{n-mp}} &\textup{if}\;\; 1\leq p<\frac nm
\\
e^{t^{\frac n{n-m-\alpha m}}} &\textup{if}\;\;p=\frac nm,\, \alpha
<\frac nm -1
\\
e^{e^{t^{\frac {n}{n-m}}}} &\textup{if}\;\; p=\frac nm,\, \alpha
=\frac nm-1
\\ \infty & \hbox{
otherwise
%if either $ p>\frac nm$, or $p=\frac nm$ and $ \alpha
%>\frac nm -1$\,
}
\end{cases} \quad \quad \hbox{near infinity,} \qquad G(t) \approx t^p \quad  \hbox{near zero.}
\end{equation*}
%near infinity, and
%$$G(t) \approx t^p$$
%near $0$.
Moreover, $L^G(\rn)$ is  optimal among all Orlicz spaces.
\\ Furthermore, by Theorem \ref{T:orlicz1} and \cite[Example 3.10]{Ci3},
\begin{equation*}%\label{orlicz10}
W^{m}L^p(\log L)^\alpha(\rn) \to \Lambda ^E (v) (\rn)\,,
\end{equation*}
where $v$ is given by \eqref{v}, and
$$
E(t) \approx \begin{cases} t^p (\log t)^\alpha &\textup{if}\; \;
1\leq p<\frac nm
\\
t^{\frac nm}(\log t)^{\alpha -\frac nm} &\textup{if} \;\; p=\frac
nm, \, \alpha <\frac nm -1
\\
t ^{\frac nm} (\log t)^{-1} (\log (\log t)) ^{-\frac nm}&\textup{if}
\; \; p=\frac nm, \, \alpha =\frac n m-1
\end{cases} \quad  \hbox{near infinity,} \quad E(t) \approx t^p \quad  \hbox{near zero.}
$$
%near infinity, and
%$$E(t) \approx t^p $$
%near $0$.
Moreover, the Orlicz-Lorentz space $\Lambda ^E(v)$ is optimal among
all rearrangement-invariant spaces. The optimal
rearrangement-invariant target space when
 either $ p>\frac nm$, or $p=\frac nm$ and $ \alpha
>\frac nm -1$ agrees with the Orlicz space $L^G(\rn)$ with $G$
defined as above, and hence with $(L^\infty \cap L^p(\log L)^\alpha)
(\rn)$.

\iffalse
\\ Note that embedding \eqref{orlicz10} is equivalent to Lorentz-Zygmund and generalized Lorentz-Zygmund spaces have not been defined yet
\begin{equation*}
W^{m}L^p(\log L)^\alpha(\rn) \to \begin{cases} L^{\frac{np}{n-mp},
p; \frac \alpha p}(\rn) \cap L^p(\log L)^\alpha(\rn)&\textup{if}\;
\; 1\leq p<\frac nm
\\
L^{\infty, \frac nm; \frac {\alpha m}{n}-1}(\rn) \cap L^p(\log
L)^\alpha(\rn) &\textup{if} \;\; p=\frac nm, \, \alpha <\frac nm -1
\\
L^{\infty, \frac nm; -\frac nm, -1}(\rn) \cap L^p(\log
L)^\alpha(\rn) &\textup{if} \; \; p=\frac nm, \, \alpha =\frac n
m-1\,.
\end{cases}
\end{equation*}

\fi

\end{example}

%%%%%%%%%%%%%%%%%%%%%%%%%%%%%%%%%%% Example 2 %%%%%%%%%%%%%%%%%%%%%%%%%%%%%%%
\iffalse
\begin{example}\label {example2} Let $W^{m,A}(\rn)$ be
as in Example \ref{example1}, namely with $m=1$ and $A$  equals
$t^p\log ^\alpha(1+t)$ for large t, where either $p = 1$ and
$\alpha\geq 0$, or $p > 1$ and $\alpha\in \R$. Then, thanks to
Theorem \ref{T:orlicz2} and \cite[Example 1.2]{Ci2}, the optimal
Orlicz target space $L^{A_{\rm opt}}(\rn)$ for embeddings of
$W^{1,A}(\rn)$ equals the Orlicz space associated with the Young
function $A_{\rm opt}$ given by
\begin{equation*}
A_{\rm opt}(t) \approx \begin{cases} A_{n}(t) &\textup{near
infinity}
\\
t^p\log ^\alpha(1+t) &\textup{near}\ 0\, ,
\end{cases}
\end{equation*}
where
\begin{equation*}
A_n(t) \approx \begin{cases} t^{p^*}\log^{\frac {n\alpha}{n-p}} (1
+t) &\textup{if}\;\; 1\leq p<n
\\
e^{t^{\frac n{n-1-\alpha}}} -1 &\textup{if}\;\;p=n,\, \alpha <n-1
\\
e^{e^{t^{n'}}} -e &\textup{if}\;\; p=n,\, \alpha =n-1
\end{cases}
\end{equation*}
near infinity.
\end{example}

\fi
\bigskip
%%%%%%%%%%%%%%%%%%%%%%%%%%%%%%%%%%%  Lubos' example A   %%%%%%%%%%%%%%%%%%%

We conclude this section with an optimal embedding theorem for
Lorentz-Sobolev spaces in $\rn$, which relies upon Theorem
\ref{T:main-higher}.

\begin{theorem}[\textbf{Lorentz-Sobolev embeddings}]\label{T:application-to-lorentz-spaces}
\textup{(i)} Let $m<n$, and either $p=q=1$, or $1<p <\frac nm$ and
$1\leq q\leq\infty$. Then
\begin{equation*}%\label{E:embeddings-to-lorentz-spaces-a}
W\sp mL\sp{p,q}(\mathbb R\sp n)\to\Lambda\sp q(w)(\mathbb R\sp n),
\end{equation*}
where
\[
w(s)=
\begin{cases}
s\sp{\frac1p-\frac1q-\frac mn} & \hbox{\rm if}\;\; s\in(0,1]\\
s\sp{\frac1p-\frac1q} &  \hbox{\rm if}\;\; s\in (1,\infty).
\end{cases}
\]
\\
\textup{(ii)} Let $m<n$ and $1< q\leq\infty$. Then
\begin{equation*}%\label{E:embeddings-to-lorentz-spaces-b}
W\sp mL\sp{\frac nm,q}(\mathbb R\sp n)\to\Lambda\sp q(w)(\mathbb
R\sp n),
\end{equation*}
where
\[
w(s)=
\begin{cases}
s\sp{-\frac1q}(1+\log\frac1s)\sp{-1}  &  \hbox{\rm if}\;\; s\in(0,1]\\
s\sp{\frac mn-\frac1q} &  \hbox{\rm if}\;\; s\in (1,\infty).
\end{cases}
\]
\\
\textup{(iii)} Let either $m>n$, or $m\leq n$ and $p>\frac nm$, or
$m\leq n$, $p=\frac nm$ and $q=1$. Then
\begin{equation*}%\label{E:embeddings-to-lorentz-spaces-c}
W\sp mL\sp{p,q}(\mathbb R\sp n)\to \Lambda\sp{E}(w)(\rn),
\end{equation*}
where
\[
w(s)=
\begin{cases}
1 &  \hbox{\rm if}\ s\in(0,1]\\
s\sp{\frac 1p-\frac1q} &  \hbox{\rm if}\ s\in (1,\infty),
\end{cases}
\]
and
\[
E(t)=
\begin{cases}
t\sp{q} &  \hbox{\rm if}\;\; t\in(0,1]\\
\infty &  \hbox{\rm if}\;\; t\in (1,\infty).
\end{cases}
\]
Moreover, in each case, the target space is  optimal  among all
rearrangement-invariant~spaces.
\end{theorem}

\section{Proofs of the main results}\label{S:proofs}
 Let $\|\cdot \|_{X(0,\infty)}$ be a rearrangement-invariant function norm, and let $\|\cdot \|_{{\widetilde X}(0,1)}$ be its localized rearrangement-invariant function norm on $(0,1)$ given  by \eqref{tilde}.
 The rearrangement-invariant function norm   $\|\cdot \|_{Z(0,1)}$
defined as in \eqref{Z} is the optimal one for which the Hardy type
inequality
\begin{equation}\label{hardyopt}
\bigg\| \int _s^1 f(r) r^{-1+\frac mn}\, dr \bigg\|_{Z(0,1)} \leq C
\|f\|_{\widetilde X(0,1)}
\end{equation}
holds with some positive constant $C$ independent of nonnegative
functions $f \in \widetilde X(0, 1)$, see \cite[Theorem~A]{T2}.
\iffalse

. Indeed, (cf.~also \cite[Theorem~A]{T2}), suppose that
\[
\bigg\| \int _s^1 f(r) r^{-1+\frac mn}\, dr \bigg\|_{S(0,1)} \leq C
\|f\|_{\widetilde X(0,1)}
\]
for some rearrangement-invariant norm $\|\cdot\|_{S(0,1)}$. Then,
using the Fubini theorem and the fact that, for every nonnegative
$f$, the function $s\mapsto \int _s^1 f(r) r^{-1+\frac mn}\, dr$ is
nonincreasing, we get
\begin{align*}
C&\geq
\sup_{\|f\|_{\widetilde X(0,1)}\leq 1, f\geq 0}\bigg\| \int _s^1 f(r) r^{-1+\frac mn}\, dr \bigg\|_{S(0,1)}\\
&= \sup_{\|f\|_{\widetilde X(0,1)}\leq 1, f\geq 0}
\sup_{\|g\|_{S'(0,1)}\leq 1}\int_0\sp 1g\sp*(s)\int _s^1 f(r) r^{-1+\frac mn}\, dr\,ds\\
&= \sup_{\|g\|_{S'(0,1)}\leq 1} \sup_{\|f\|_{\widetilde X(0,1)}\leq
1, f\geq 0}
\int_0\sp 1f(r)r^{\frac mn}g\sp{**}(r)\, dr\\
&= \sup_{\|g\|_{S'(0,1)}\leq 1} \|r^{\frac
mn}g\sp{**}(r)\|_{\widetilde X'(0,1)}= \sup_{\|g\|_{S'(0,1)}\leq
1}\|g\|_{Z'(0,1)},
\end{align*}
showing $S'(0,1)\to Z'(0,1)$, that is, $Z(0,1)\to S(0,1)$. This
proves the optimality property of the norm $\|\cdot\|_{Z(0,1)}$.

 \fi
 The following
 proposition tells us that such a norm  is always at least as strong as that of ${\widetilde X}(0,1)$.

\begin{proposition}\label{ZX} Let $n\geq 2$ and let $m\in \mathbb N$. Assume
that  $\| \cdot \|_{X(0,\infty)}$ is a rearrangement-invariant
function norm, and let $\| \cdot \|_{Z(0,1)}$ be
rearrangement-invariant function norm defined as in \eqref{Z}. Then
\begin{equation}\label{a1}
Z(0,1) \to {\widetilde X}(0,1).
\end{equation}
\begin{proof}
Embedding \eqref{a1} is equivalent to
\begin{equation}\label{a2}
{\widetilde X}'(0,1) \to Z'(0,1).
\end{equation}
By the very definition of $Z(0,1)$, embedding \eqref{a2} is in turn
equivalent to the inequality
\begin{equation}\label{a3}
\|s\sp{\frac mn}f\sp{**}(s)\|_{\widetilde X'(0,1)}\leq C
\|f\|_{{\widetilde X}'(0,1)}
\end{equation}
for some constant $C$ and for every $f \in \mathcal M_+(0,1)$. It is
easily verified that the operator
$$ \mathcal M_+(0,1) \ni f(s) \mapsto s\sp{\frac mn}f\sp{**}(s) \in \mathcal M_+(0,1) $$
is sublinear, and bounded in $L^1(0,1)$ and in $L^\infty (0,1)$. An
interpolation theorem by Calder\'on \cite[Theorem 2.12, Chapter
3]{BS} then tells us that it is bounded in any
rearrangement-invariant space. Hence, inequality \eqref{a3} follows.
\end{proof}

\end{proposition}

%Given  two rearrangement-invariant function norms $\|\cdot\|_{X(0,\infty)}$ and
%$\|\cdot\|_{P(0,\infty)}$, define the functional $\|\cdot\|_{Y(0,\infty)}$ by
%\begin{equation}\label{E:definition-of-Y}
%\|f\|_{Y(0,\infty)}=\|f\|_{X(0,\infty)}+\|f\sp*\chi_{(0,1)}\|_{P(0,\infty)}
%\end{equation}
%for $f \in \mathcal M_+(0, \infty)$.
%One can verify that $\|\cdot\|_{Y(0,\infty)}$ is
%also a~rearrangement-invariant function norm.

The next auxiliary result will be critical in the proof of the
sharpness of our embeddings. Throughout, we shall use the relation
$\gtrsim$ [$\lesssim$] between two expressions to denote that the
former bounds [is bounded by] the latter, up to a positive constant.
The notation $\simeq$ is adopted to denote that the relations
$\gtrsim$ and $\lesssim$ hold simultaneously. Notice the different
meanings of the relation $\simeq$ and the relation $\approx$ of
equivalence between Young functions introduced in Section
\ref{S:preliminaries}.

\begin{proposition}\label{P:1} Let
$\|\cdot\|_{X(0,\infty)}$ and $\|\cdot\|_{S(0,\infty)}$ be
rearrangement-invariant function norms. Let $\alpha>0$ and let
$L>0$. Assume that
\begin{equation}\label{E:1.23_bis}
\left\|\chi_{(L,\infty)}(s)\int_s\sp{\infty}f(r)r\sp{\alpha-1}\,dr\right\|_{S(0,\infty)}
\leq C_1
\left(\|f\|_{X(0,\infty)}+\left\|\int_s\sp{\infty}f(r)r\sp{\alpha-1}\,dr\right\|_{X(0,\infty)}\right)
\end{equation}
for some constant $C_1$, and for every non-increasing function
$f\in\M_+(0,\infty)$. Then there exists a constant $C_2$ such that
\begin{equation}\label{E:assertion-of-proposition_bis}
\|g\sp*\chi_{(2L,\infty)}\|_{S(0,\infty)} \leq
C_2\|g\|_{X(0,\infty)}
\end{equation}
for every $g\in\M_+(0,\infty)$.
\end{proposition}

 \begin{proof}
Let $f$ be as in the statement. Assume, in addition, that $f$ is
constant on $(0,2L)$, and denote by $f_0\in \R$ the constant value
of $f$ on $(0, 2L)$. If $s\geq 2L$, then
\begin{equation*}
\int_s^{\infty} f(r) r^{\alpha -1} \; dr \geq \int_s^{2s} f(r)
r^{\alpha -1} \; dr \geq f(2s) \int_s^{2s} r^{\alpha -1} \; dr =
(2^{\alpha} -1) \frac {s^\alpha}{\alpha} f(2s)\geq (2^{\alpha} -1)
\frac{(2L)^\alpha}{\alpha} f(2s)\,.
\end{equation*}
Furthermore, if $s\in \left(0, L\right)$, then
\begin{equation*}
\int_s^{\infty} f(r) r^{\alpha -1} \; dr \geq \int_{L}^{2L} f(r)
r^{\alpha -1} \; dr = \frac 1\alpha \left[(2L)^\alpha - L^\alpha
\right] f_0\,.
\end{equation*}
Thus,
\begin{align}\label{E:1.24_bis}
\|f(s)\|_{X(0, \infty)}& \lesssim \|f(2s)\|_{X(0, \infty)}\leq
\|f(2s) \chi_{(2L, \infty)}(s)\|_{X(0, \infty)} + \|f(2s) \chi_{(0,
2L)}(s)\|_{X(0, \infty)}
\\
\cr & \lesssim \bigg \|\chi_{(2L, \infty)} (s) \int_s^\infty f(r)
r^{\alpha -1} \; dr \bigg \|_{X(0, \infty)} + f_0 \big \| \chi_{(0,
2L)} \big\|_{X(0, \infty)} \nonumber
\\
\cr & \lesssim \bigg \|\chi_{(2L, \infty)} (s) \int_s^\infty f(r)
r^{\alpha -1} \; dr \bigg \|_{X(0, \infty)} + f_0 \big \| \chi_{(0,
L)} \big\|_{X(0, \infty)} \nonumber
\\
\cr & \lesssim \bigg \|\chi_{(2L, \infty)} (s) \int_s^\infty f(r)
r^{\alpha -1} \; dr \bigg \|_{X(0, \infty)} + \bigg \| \chi_{(0,
L)}(s) \int_s^\infty f(r) r^{\alpha -1} \; dr \bigg\|_{X(0, \infty)}
\nonumber
\\
\cr & \lesssim \bigg \|\int_s^\infty f(r) r^{\alpha -1} \; dr \bigg
\|_{X(0, \infty)} \,.\nonumber
\end{align}
Inequalities \eqref{E:1.23_bis} and \eqref{E:1.24_bis} imply that
\begin{equation}\label{E:1.25_bis}
\bigg \| \chi_{(L, \infty)}(s) \int_s^\infty f(r) r^{\alpha -1} \;
dr\bigg\|_{S(0, \infty)} \lesssim \bigg \| \int_s^\infty f(r)
r^{\alpha -1} \; dr\bigg\|_{X(0, \infty)}\,.
\end{equation}
Consider functions $f$ of the form
\begin{equation*}
f=\sum_{i=1}^k a_i \chi(0, b_i) \qquad \hbox{with $a_i\geq 0$ \;\;
and \;\; $b_i \geq 2L \,.$}
\end{equation*}
For this choice of $f$, one has that
\begin{equation*}
\int_s^\infty f(r) r^{\alpha -1}\; dr  = \sum_{i=1}^k \chi_{(0,
b_i)} (s) \,\frac{a_i}\alpha \,\left( b_i^\alpha - s^\alpha\right)
\qquad \hbox{for $s>0$},
\end{equation*}
whence
\begin{equation*}%\label{E:1.26_bis}
\bigg \|\int_s^\infty f(r) r^{\alpha -1}\; dr\bigg\|_{X(0, \infty)}
\lesssim\bigg\|\sum_{i=1}^k \chi_{(0, b_i)} (s) \, a_i b_i^\alpha
\bigg\|_{X(0, \infty)}\,.
\end{equation*}
On the other hand,
\begin{align}\label{E:1.27_bis}
&\bigg\|\chi_{(L, \infty)}(s)\int_s^\infty f(r) r^{\alpha -1}\;
dr\bigg\|_{S(0, \infty)} \gtrsim \bigg\|\chi_{(L,
\infty)}(s)\sum_{i=1}^k \chi_{\big(0, \frac{b_i}2\big)} (s) a_i
(b_i^\alpha - s^\alpha)\bigg\|_{S(0, \infty)}
\\
\cr & \gtrsim \bigg\|\chi_{(L, \infty)}(s)\sum_{i=1}^k \chi_{\big(0,
\frac{b_i}2\big)} (s) a_i \left(b_i^\alpha - \left(
{b_i}/2\right)^\alpha\right)\bigg\|_{S(0, \infty)} \gtrsim
\bigg\|\chi_{(L, \infty)}(s)\sum_{i=1}^k \chi_{\big(0,
\frac{b_i}2\big)} (s) a_i b_i^\alpha\bigg\|_{S(0, \infty)}\nonumber
\\
\cr & \simeq \bigg\|\chi_{(L, \infty)}\big( s/2\big)\sum_{i=1}^k
\chi_{\big(0, \frac{b_i}2 \big)} \left( s/2\right) a_i
b_i^\alpha\bigg\|_{S(0, \infty)}\simeq \bigg\|\chi_{(2L,
\infty)}(s)\sum_{i=1}^k \chi_{(0, b_i)} (s) a_i
b_i^\alpha\bigg\|_{S(0, \infty)}\, .\nonumber
\end{align}
Inequalities \eqref{E:1.25_bis} and \eqref{E:1.27_bis} imply, via an
approximation argument, that
\begin{equation}\label{E:1.28_bis}
\|\chi_{(2L, \infty)}\, g \|_{S(0, \infty)} \lesssim \|g\|_{X(0,
\infty)}
\end{equation}
for every non-increasing function $g\in \mathcal{M}_+ (0, \infty)$
that is constant on $(0, 2L)$. Observe that property \textup{(P3)}
of rearrangement-invariant function norms plays a role in the
approximation in question.
\par
Assume now that $g$ is any function in $\mathcal{M}_+ (0, \infty)$.
An application of \eqref{E:1.28_bis} with $g$ replaced by $g^* (2L)
\chi_{(0, 2L)} + g^* \chi_{(2L, \infty)}$ yields
\begin{align}\label{E:1.29_bis}
\big\| \chi_{(2L, \infty)} g^* \big\|_{S(0, \infty)} & = \big\|
\chi_{(2L, \infty)} g^* (2L) \chi_{(0, 2L)} + g^* \chi_{(2L,
\infty)}\big\|_{S(0, \infty)}
\\
\cr & \lesssim \big\| g^* (2L) \chi_{(0, 2L)} + g^* \chi_{(2L,
\infty)}\big\|_{X(0, \infty)}\leq \|g^*\|_{X(0, \infty)} =
\|g\|_{X(0, \infty)}\,. \nonumber
\end{align}
Hence, inequality \eqref{E:assertion-of-proposition_bis} follows.
\end{proof}

%%%%%%%%%%%%%%%%%%%%%%%%%%%%%%%%%%%%%%%%%

A key step towards  the general embedding theorem  of Theorem
\ref{T:main-higher} is a first-order embedding, for somewhat more
general non-homogeneous Sobolev type spaces defined as follows.
Given two rearrangement-invariant function norms
$\|\cdot\|_{X(0,\infty)}$ and $\|\cdot\|_{Y(0,\infty)}$, we define
the space $W\sp 1(X,Y)(\rn)$ as
$$W^1(X,Y)(\rn) = \{ u\in X(\rn): u \;\; \hbox{is weakly differentiable in $\rn$, and  $|\nabla u|\in
Y(\rn)$}\}$$
% the set of weakly differentiable functions $u:\rn\to \R$
%such that $u\in X(\rn)$ and $|\nabla u|\in Y(\rn)$,
endowed with the norm
$$
\|u\|_{W^1(X,Y)(\rn)} = \|u\|_{X(\rn)} +\|\nabla u\|_{Y(\rn)}.
$$
\iffalse The proof of the relevant first-order embedding relies upon
a Poincar\'e type inequality for weakly differentiable functions in
$\rn$ with support of finite measure. This inequality asserts that
there exists a constant $C$ such that
\begin{equation}\label{reduction1}
\|u\|_{Y^1(\rn)}\leq C \|\nabla u\|_{Y(\rn)}
\end{equation}
for every function $u$ such that $|{\rm supp} u| \leq 1$ and
$|\nabla u| \in Y(\rn)$. Here, $Y^1(\rn)$ denotes the
rearrangement-invariant space defined as in
\eqref{E:definition-of-Y-opt-2-higher}, with $m=1$ and $X(\rn)$
replaced by $Y(\rn)$.
%, and is optimal  in \eqref{reduction1} among
%all rearrangement-invariant spaces.
Inequality  \eqref{reduction1} is  proved in \cite{???} in a
slightly different form, namely for functions defined in an open set
of finite measure, that vanish  on its boundary. The inequality
continues to hold, with the same proof, also in the present form.

\fi

\smallskip

\begin{theorem}[\textbf{Non-homogeneous first-order  Sobolev embeddings}]\label{T:first-order}
Let $\|\cdot\|_{X(0,\infty)}$ and $\|\cdot\|_{Y(0,\infty)}$ be
rearrangement-invariant function norms. Let
$\|\cdot\|_{Y\sp{1}(0,\infty)}$ be defined as
in~\eqref{E:definition-of-Y-opt-2-higher}, with $m=1$ and $X(\rn)$
replaced by $Y(\rn)$. Then
\begin{equation}\label{E:sobolev-opt}
W\sp{1}(X,Y)(\rn)\hra (Y^1\cap X)(\rn)\,.
\end{equation}
% Moreover, the range space is optimal in \eqref{E:sobolev-opt} among all
%rearrangement-invariant spaces.
\end{theorem}

\begin{proof}
We begin by showing that there exists a constant $C$ such that
\begin{equation}\label{reduction1}
\|u\|_{Y^1(\rn)}\leq C \|\nabla u\|_{Y(\rn)}
\end{equation}
for every function $u$ such that $|{\rm supp} u| \leq 1$ and
$|\nabla u| \in Y(\rn)$.
%Here, $Y^1(\rn)$ denotes the
%rearrangement-invariant space defined as in
%\eqref{E:definition-of-Y-opt-2-higher}, with $m=1$ and $X(\rn)$
%replaced by $Y(\rn)$.
Indeed, a general form of the P\'olya-Szeg\"o principle on the
decrease of gradient norms under spherically symmetric
symmetrization ensures that, for any such function $u$, the
decreasing rearrangement $u^*$ is locally absolutely continuous in
$(0,\infty)$, and
\begin{equation}\label{PS}
\|\nabla u\|_{Y(\rn)} \geq n \omega _n^{\frac 1n}\big\|s^{\frac
1{n'}}{u^*}'(s)\big\|_{Y(0,\infty)}\,,
\end{equation}
where $\omega _n$ denotes the Lebesgue measure of the unit ball in
$\rn$ -- see e.g. \cite[Lemma 4.1]{CParkiv}. Moreover, there exists
a constant $C$ such that
\begin{align}\label{may21}
\big\|s^{\frac 1{n'}}{u^*}'(s)\big\|_{Y(0,\infty)} = \big\|s^{\frac
1{n'}}{u^*}'(s)\big\|_{\widetilde Y(0,1)} \geq C \bigg\|\int _s^1
{u^*}'(r)\, dr\bigg\|_{Z(0,1)} = C \|u^*(s)\|_{Z(0,1)} =
\|u\|_{Y^1(\rn)},
\end{align}
where the first equality holds since $u^*$ vanishes in $(1,
\infty)$, the inequality is a consequence of the Hardy type
inequality \eqref{hardyopt}, the second equality holds since
$u^*(1)=0$, and the last inequality by the very definition of the
norm in $Y^1(\rn)$. Inequality \eqref{reduction1} is a consequence
of \eqref{PS} and \eqref{may21}.
\\ Now, let $u\in W^1 (X,Y)(\rn)$ such that
\begin{equation}\label{7a}
\|u\|_{W\sp1(X, Y)(\rn)} \leq 1.
\end{equation}
Then,
\begin{equation}\label{E:3}
1\geq \|u\|_{X(\rn)}\geq t \,\|\chi_{\{|u|>t\}}\|_{X(\rn)}= t\,
\varphi_X (|\{|u|>t\}|)\qquad\textup{for}\ t>0,
\end{equation}
where $\varphi _X$ is the fundamental function of $\|\cdot\|_{X(0,
\infty)}$ defined as in \eqref{fund}. Let $\varphi_X^{-1}$ denote
its generalized left-continuous inverse. One has that
\begin{equation*}%\label{5a}
\lim _{t \to 0^+} \varphi_X^{-1} (t) =0,
\end{equation*}
and
\begin{equation*}%\label{6a}
\varphi_X^{-1} (\varphi _X (t)) \geq t \qquad \hbox{for $t>0$.}
\end{equation*}
Therefore, by \eqref{E:3},
$$
|\{|u|>t\}|\leq \varphi_X^{-1}\big(\varphi _X (|\{|u|>t\}|)\big)
\leq \varphi_X^{-1}(1/t)\qquad\textup{for}\ t>0.
$$
 Hence,
\[
|\{|u|>t\}|\leq \varphi_X^{-1}(1/t)\qquad\textup{for}\ t>0.
\]
Choose $t_0$ such that $\varphi^{-1}_X( 1/{t_0})\leq 1$, whence
\[
|\{|u|>t_0\}|\leq 1.
\]
Let us decompose the function $u$ as $u=u_1 + u_2$, where
$u_1=\operatorname{sign}(u)\min\{|u|,t_0\}$ and $u_2= u-u_1$. By
standard properties of truncations of Sobolev functions, we have
that $u_1, u_2 \in W^1(X,Y)$. Moreover, $\|u_1\|_{L^\infty
(\rn)}\leq t_0$, and $|\{|u_2|>0\}|=|\{|u|>t_0\}|\leq 1$. We claim
that there exist positive constants $c_1$ and $c_2$ such that
\begin{align}
\|u_1\|_{Y\sp1(\rn)}\leq c_1,\label{E:4a}
\end{align}
and
\begin{align}
\|u_2\|_{{Y}\sp1(\rn)}\leq c_2.\label{E:4b}
\end{align}
Inequality \eqref{E:4a} is a consequence of the fact that
\[
\|u_1\|_{Y\sp1(\rn)}\leq t_0\,\|1\|_{Y\sp1(\rn)}<\infty\,,
\]
where the last inequality holds thanks to the definition of the norm
$\|\cdot\|_{Y\sp1(\rn)}$.
%We observe that due to the localized nature of the space
%$Y\sp1(\Rn)$ one has $\|1\|_{Y\sp1(\Rn)}<\infty$.
Inequality\eqref{E:4b} follows from inequalities \eqref{reduction1}
and \eqref{7a}.
%Next, if $|\{|u_2|>0\}|\leq1$, then
%$|\operatorname{supp}(u_2)|<\infty$. Hence, applying the well-known
%reduction principle
%\[
%\left\|\int_t\sp
%1f(s)s\sp{-1+\alpha}\,ds\right\|_{\widetilde{Y}\sp1(0,1)}\leq
%C\|f\|_{\widetilde{Y}(0,1)}
%\]
%which is valid for every $f\in\Mpl(0,1)$,to the particular choice
%$f=|\nabla u_2|\sp*$, we obtain~\eqref{E:4b}.
From  \eqref{E:4a}, \eqref{E:4b} and \eqref{7a} we infer that
\begin{align*}
\|u\|_{X(\rn)}+\|u\|_{{Y}\sp1(\rn)} &\leq
\|u\|_{X(\rn)}+\|u_1\|_{{Y}\sp1(\rn)} +\|u_2\|_{{Y}\sp1(\rn)} \leq
\|u\|_{X(\rn)} + c_1+c_2 \leq 1+c_1+c_2
\end{align*}
for every function $u$ fulfilling \eqref{7a}. This establishes
embedding \eqref{E:sobolev-opt}.
\end{proof}

We are now in a position the accomplish the proof of
Theorem~\ref{T:main-higher}.

\begin{proof}[Proof of Theorem~\ref{T:main-higher}]
We shall prove by induction on $m$ that
\begin{equation}\label{E:induction}
\|u\|_{X^m(\Rn)}\leq C \|u\|_{W^mX(\Rn)}
\end{equation}
for some constant $C$, and for every $u\in W^mX (\Rn)$. This
inequality, combined with the trivial embedding $W^mX(\Rn)
\rightarrow X(\Rn)$,  yields embedding~\eqref{E:sobolev-opt-higher}.
\\
If $m=1$, then inequality~\eqref{E:induction} is a straightforward
consequence of Theorem~\ref{T:first-order} applied in the special
case when $Y(\Rn)=X(\Rn)$. Now, assume that~\eqref{E:induction} is
fulfilled for some $m\in \mathbb N$. Let $u\in W^{m+1}X(\Rn)$. By
induction assumption applied to the function $u_{x_i}$ for
$i=1,\dots,n$,
$$
\|u_{x_i}\|_{X^m(\Rn)}\leq C\|u_{x_i}\|_{W^mX(\Rn)} \leq
C'\|u\|_{W^{m+1}X(\Rn)}
$$
for some constants $C$ and $C'$. Consequently,
\begin{equation}\label{april3}
\|\nabla u\|_{X^m(\Rn)}\leq  C\|u\|_{W^{m+1}X(\Rn)}
\end{equation}
for some constant $C$. From Theorem~\ref{T:first-order} with
$Y(\Rn)=X^m(\Rn)$, and inequality \eqref{april3}, we obtain
\begin{equation}
\label{april4} \|u\|_{(X^m)^1(\Rn)} \leq
C\left(\|u\|_{X(\Rn)}+\|\nabla u\|_{X^m(\Rn)}\right) \leq
C'\|u\|_{W^{m+1}X(\Rn)}
\end{equation}
for some constants $C$ and $C'$. From~\cite[Corollary~9.6]{CPS} one
can deduce that
\begin{equation}
\label{april5} \|u\|_{(X^m)^1(\Rn)} \simeq \|u\|_{X^{m+1}(\Rn)}
\end{equation}
 for every $u\in W^{m+1}X(\Rn)$.  Inequality~\eqref{E:induction}, with $m$ replaced by $m+1$, follows from \eqref{april4} and \eqref{april5}.
\\
It remains to prove the optimality of the space $( X\sp m\cap
X)(\rn)$. Assume that $S(\rn)$ is another rearrangement-invariant
space such that
\[
W\sp{m}X(\rn)\hra S(\rn).
\]
Then there exists a constant $C$ such that
\begin{equation}\label{E:10}
\|u\|_{S(\rn)} \leq  C\| u \|_{W^m X(\rn)}
\end{equation}
for every $u\in W^m X(\rn)$. We have to show that
$$(X\cap X\sp m)(\rn) \to S(\rn),$$
or, equivalently, that
\begin{equation}\label{12a}
\|f\|_{S(0,\infty)} \leq C \left(\|f\|_{X(0,\infty)}+ \|f\|_{{X}\sp
m(0,\infty)}\right)
\end{equation}
for some constant $C$, and every $f \in \mathcal M_+(0, \infty)$.
%
%namely,
%\[
%\|f\|_{S(0,\infty)} \leq C \|f\|_{(X\cap Y\sp1)(0,\infty)},
%\]
Inequality \eqref{12a} will follow if we  show that
\begin{equation}\label{E:claim-1}
\|f\sp*\chi_{(0,1)}\|_{S(0,\infty)} \leq C \|f\|_{{X}\sp
m(0,\infty)}
\end{equation}
and
\begin{equation}\label{E:claim-2}
\|f\sp*\chi_{(1,\infty)}\|_{S(0,\infty)} \leq C \|f\|_{X(0,\infty)}
\end{equation}
for some constant $C$, and for every  $f \in \mathcal M_+(0,
\infty)$.
\\
Let $\mathcal B$ be the ball in $\rn$, centered at $0$, such that
$|\mathcal B|=1$. We claim that inequality \eqref{E:10} implies that
\begin{equation}\label{A1}
\|v\|_{S(\mathcal B)} \leq C \|v\|_{W^m X(\mathcal B)}
\end{equation}
for some constant $C$, and for every $v\in W^m X(\mathcal B)$. Here,
$X(\mathcal B)$ and $S(\mathcal B)$ denote the
rearrangement-invariant spaces built upon the function norms $\|
\cdot \|_{\widetilde{X}(0, 1)}$ and $\| \cdot \|_{\widetilde{S}(0,
1)}$ defined as in \eqref{tilde}. To verify this claim, one can make
use of the fact that there exists a bounded extension operator $T:
W^m X(\mathcal B) \hra W^m X(\rn)$ \cite[Theorem 4.1]{CR}. Thus,
\begin{equation}\label{A2}
\|T v\|_{W^m X(\rn)} \leq C \|v\|_{W^m X(\mathcal B)}
\end{equation}
for some constant $C$ and for every $v\in W^m X(\mathcal B)$.
Coupling \eqref{E:10} with \eqref{A2} we deduce that
\begin{equation*}%\label{A3}
\|v\|_{S(\mathcal B)} = \|T v\|_{S(\mathcal B)} \leq \|T v
\|_{S(\rn)} \leq C \| T v \|_{W^m X(\rn)} \leq C' \|  v \|_{W^m
X(\mathcal B)}
\end{equation*}
for some constants $C$ and $C'$, and for every $v\in W^m X(\mathcal
B)$. Hence, inequality \eqref{A1} follows.
\\
 Inequality \eqref{A1} in turn implies
that
\begin{equation}\label{E:newly-proved-necessity}
\bigg\| \int _s^1 g(r) r^{-1+\frac mn}\, dr \bigg\|_{\widetilde
S(0,1)} \leq C \|g\|_{\widetilde X(0,1)}
\end{equation}
for some positive constant $C$, and every $g\in\M_+(0,1)$.
 This implication can be found  in the proof of
\cite[Theorem~A]{T2}. For completeness, we provide a proof
hereafter, that also fixes some details in that of \cite{T2}.
%or \cite[Theorem 6.2]{CPS}).
\\
Let us preliminarily note that  we can restrict our attention to the
case when $m<n$. Indeed, if $m \geq n$, inequality
\eqref{E:newly-proved-necessity} holds with $\|\cdot \|_{\widetilde
S(0,1)} = \|\cdot \|_{L^\infty (0, 1)}$, and hence for every
rearrangement-invariant norm $\|\cdot \|_{S(0,\infty)}$.
\\
Given any bounded function $f \in \mathcal M_+(0,1)$, define
\begin{equation}\label{31a}
u(x)= \int_{\omega_n|x|^n}^{1} \int_{s_1}^{1} \int_{s_2}^{1}\cdots
\int_{s_{m-1}}^{1} f(s_m) s_m^{-m+ \frac mn} \; ds_m\cdots ds_1
\qquad \hbox{for $x \in \mathcal B$.}
\end{equation}
%where $\omega_n$ denotes the  Lebesgue measure of the
%unit ball $\mathbb B^n$ in $\rn$.
Set $M=\omega_n\sp{-\frac 1n}$. We need to derive a pointwise
estimate for $|D^mu|$. As a preliminary step, consider  any function
$v : \mathcal B \to [0, \infty)$ given by
\begin{equation*}
v(x)= g(|x|) \qquad \hbox{for $x \in \mathcal B$,}
\end{equation*}
where   $g: (0,M)\to[0,\infty)$ is an $m$-times weakly
differentiable function. One can show that every $\ell$-th order
derivative of $v$, with $1 \leq \ell \leq m$, is a~linear
combination of terms of the form
\[
\frac{x_{\alpha_1}\dots x_{\alpha_i}g\sp{(j)}(|x|)}{|x|\sp k} \qquad
\hbox{for a.e. $x \in \mathcal B$,}
\]
where  $\alpha_1,\dots,\alpha_i\in\{1,\dots,n\}$, and
\begin{equation*}%\label{E:one}
1\leq j\leq\ell,\quad 0\leq i\leq \ell, \quad k-i=\ell-j.
\end{equation*}
Here, $g\sp{(j)}$ denotes the $j$-th order derivative of $g$. As a
consequence,
\begin{equation}\label{E:doublesum}
\sum_{\ell=1}\sp{m}|\nabla\sp{\ell}v(x)|\leq
C\sum_{\ell=1}\sp{m}\sum_{k=1}\sp{\ell}\frac{|g\sp{(k)}(|x|)|}{|x|\sp{\ell-k}}
\qquad \hbox{for a.e. $x \in \mathcal B$.}
\end{equation}
Next, consider functions $g$  defined by
$$
g(s)=\int_{\omega_n s\sp
n}\sp{1}\int_{s_1}\sp{1}\dots\int_{s_{m-1}}\sp{1}
f(s_m)s_m\sp{-m+\frac mn}\,ds_m\dots ds_1 \qquad \hbox{for $s \in
(0,M)$,}
$$
where $f$ is as in \eqref{31a}.  It can be verified that, for each
$1\leq k\leq m-1$, the function $g\sp{(k)}(s)$ is a linear
combination of functions of the form
\[
s\sp{jn-k}\int_{\omega_n s\sp
n}\sp{1}\int_{s_{j+1}}\sp{1}\dots\int_{s_{m-1}}\sp{1}
f(s_m)s_m\sp{-m+\frac mn}\,ds_m\dots ds_{j+1} \qquad \hbox{for $s
\in (0, M)$},
\]
where $j\in \{1,2,\dots,k\}$,
%are such that $jn-k\geq 0$,
whereas $g\sp{(m)}(s)$ is a linear combination of functions of the
form
\begin{equation}\label{26a}
s\sp{jn-m}\int_{\omega_n s\sp
n}\sp{1}\int_{s_{j+1}}\sp{1}\dots\int_{s_{m-1}}\sp{1}
f(s_m)s_m\sp{-m+\frac mn}\,ds_m\dots ds_{j+1}\qquad \hbox{for $s \in
(0,M)$,}
\end{equation}
where $j\in \{1,2,\dots,m-1\}$, and of the function $f(\omega_n s\sp
n)$. Note that, if $j=m-1$, then the expression in \eqref{26a} has
to be understood as
\[
s^{(m-1)n-m}\int_{\omega_n s\sp n}\sp{1} f(s_m)s_m\sp{-m+\frac
mn}\,ds_m\qquad \hbox{for $s \in (0,M)$.}
\]
As a consequence of these formulas, we can infer that, if $1\leq
k\leq m-1$, then
\begin{equation}\label{27a}
|g\sp{(k)}(s)|\lesssim \sum_{j=1}\sp k s\sp{jn-k}\int_{\omega_n s\sp
n}\sp{1}f(r)r\sp{-j+\frac mn-1}dr \qquad \hbox{for $s \in (0,M)$,}
\end{equation}
and
\begin{equation}\label{28a}
|g\sp{(m)}(s)|\lesssim \sum_{j=1}\sp{m-1}s\sp{jn-m}\int_{\omega_n
s\sp n}\sp{1}f(r)r\sp{-j+\frac mn-1}dr +f(\omega_n s\sp n)  \qquad
\hbox{for a.e. $s \in (0,M)$.}
\end{equation}
From equations \eqref{E:doublesum}, \eqref{27a} and \eqref{28a} one
can deduce that
\begin{align}\label{29a}
|D^m u(x)| &\lesssim f(\omega_n |x|^n) + \int_{\omega_n |x|^n}^1
f(s) s^{-1 + \frac mn}\; ds + \sum_{j=1}\sp{m-1} |x|^{jn-m}
\int_{\omega_n |x|^n}^1 f(s) s^{-j + \frac mn -1}ds\
\end{align}
for a.e. $x \in \mathcal B$. On the other hand, by Fubini's theorem,
\begin{equation}\label{E:12}
u(x) = \int_{\omega_n |x|^n}^1 f(s) s^{-m + \frac mn} \frac{(s -
\omega_n |x|^n)^{m-1}}{(m-1)!}\; ds
 \gtrsim
\chi_{(0,1)}(2\omega_n |x|^n) \int_{2\omega_n |x|^n}^{1} f(s)
s^{\frac mn -1}\; ds
\end{equation}
%where $\frac12\mathcal B$ is the ball {of measure $1/2$} centered at zero.
for  $x \in \mathcal B$. The following chain holds:
\begin{align}\label{30a}
&\left \|\int_{s}^1 f(r) r^{-1 + \frac mn} \; dr \right
\|_{\widetilde S(0,1)} \lesssim \left
\|\chi_{(0,\frac12)}(t)\int_{2s}^1 f(r) r^{-1 + \frac mn} \; dr
\right \|_{\widetilde S(0,1)} \lesssim {\|u\|_{S(\mathcal B)}}
\lesssim  \| D^m u\|_{X(\mathcal B)}
\\
& \lesssim  \| f\|_{\widetilde X(0,1)} + \left \|\int _s^1 f(r)
r^{-1 +
% * <cianchi@unifi.it> 2017-05-23T07:54:10.750Z:
%
% ^.
% * <cianchi@unifi.it> 2017-05-23T07:54:09.070Z:
%
% ^.
\frac mn}\; dr\right\|_{\widetilde X(0,1)} + \sum _{j=1}^{m-1}
\left\|s^{j- \frac mn} \int_s^1 f(r) r^{-j + \frac mn -1}\; dr
\right\|_{\widetilde X(0,1)} \nonumber
\\
& \nonumber
% \lesssim  \| f\|_{\widetilde X(0,1)} + \left \|\int _s^1 f(r) r^{-1 +
%\frac mn}\; dr\right\|_{\widetilde X(0,1)}
 \lesssim  \| f\|_{\widetilde X(0,1)}.
\end{align}
Here, we have made use of \eqref{E:12}, \eqref{E:10}, \eqref{29a},
and of  the boundedness of the operators
$$
f (s) \mapsto s^{j- \frac mn} \int_s^1 f(r) r^{-j + \frac mn -1}\;
dr \quad\textup{and}\quad f (s) \mapsto \int_s^1 f(r) r^{\frac mn
-1}\,dr
$$
on $\widetilde X (0,1)$ for every $j$ as above. The boundedness of
these operators follows from Calder\'on interpolation theorem, owing
to their boundedness in   $L\sp1(0,1)$ and $L\sp{\infty}(0,1)$.
Inequality \eqref{E:newly-proved-necessity} follows
from~\eqref{30a}.
%By the same line of argument that has been applied at the beginning of Section~\ref{S:proofs}, we conclude from~\eqref{E:newly-proved-necessity} that $Z(0,1)\to\widetilde S(0,1)$,
Inequality \eqref{E:newly-proved-necessity} implies, via the
optimality of the norm $\|\cdot \|_{Z(0,1)}$ in \eqref{hardyopt},
that
\[
\|g\|_{\widetilde S(0,1)}\lesssim \|g\|_{Z(0,1)}
\]
for every $g\in\M_+(0,1)$.
%This inequality holds in particular for every non-increasing function on $(0,1)$, so we get
%\[
%\|g\sp*\|_{\widetilde S(0,1)}\leq C\|g\sp*\|_{Z(0,1)}.
%\]
Given any~function $f\in\M_+(0,\infty)$, we can apply the latter
inequality to $f\sp*\chi_{(0,1)}$, and obtain
\[
\|f\sp*\chi_{(0,1)}\|_{\widetilde S(0,1)}\leq
C\|f\sp*\chi_{(0,1)}\|_{Z(0,1)}.
\]
By~\eqref{tilde} and~\eqref{E:definition-of-Y-opt-2-higher}, this
entails \eqref{E:claim-1}.
\\
Let us next focus on  \eqref{E:claim-2}. Fix $L>0$ and consider
trial functions in \eqref{E:10} of the form
\begin{equation*}
u(x)= \varphi (x') \, \psi(x_1) \, \int_{x_1}^\infty
\int_{s_1}^\infty \cdots \int_{s_{m-1}}^\infty f(s_m) \; ds_m \cdots
ds_1 \qquad \hbox{for $x\in \rn\,,$}
\end{equation*}
where $x=(x_1, x')$, $x'\in \R^{n-1}$, $x_1\in \R$. Here, $f\in
\mathcal{M}_+(\R)$  and has bounded support; $\psi \in
\mathcal{C}^\infty (\R)$, with $\psi=0$ in  $(-\infty, L]$,
$\psi\equiv 1$ in  $[2L, +\infty)$, $0\leq \psi  \leq 1$ in $\R$;
$\varphi \in \mathcal{C}^\infty_0 (\mathcal B_2^{n-1})$, with
$\varphi=1$ in $\mathcal B_1^{n-1}$, where $\mathcal B_\rho ^{n-1}$
denotes the ball in $\R^{n-1}$, centered at $0$, with radius $\rho$.
An application of Fubini's theorem yields
\begin{equation*}
u(x)= \varphi (x') \, \psi(x_1) \, \int_{x_1}^\infty
 f(s) \frac{(s- x_1)^{m-1}}{(m-1)!}\; ds \qquad \hbox{for $x\in \rn\,.$}
\end{equation*}
Thus,
\begin{align}\label{A4}
u(x)&\geq \chi_{\mathcal B_1^{n-1}} (x')\, \frac{\psi(x_1)}{2^{m-1}
(m-1)!} \, \int_{2x_1}^\infty f(s) s^{m-1}\; ds
\\
\cr & \geq  \frac{\chi_{\mathcal B_1^{n-1}} (x')\,\chi_{(2L,
\infty)}(x_1)}{2^{m-1} (m-1)!} \, \int_{2x_1}^\infty f(s) s^{m-1}\;
ds
 \qquad \hbox{for $x\in \rn\,,$} \nonumber
\end{align}
and
\begin{equation}\label{A5}
u(x)\leq  \frac{\chi_{\mathcal B_2^{n-1}} (x')\,\chi_{(L,
\infty)}(x_1)}{(m-1)!} \, \int_{x_1}^\infty f(s) s^{m-1}\; ds
 \qquad \hbox{for $x\in \rn\,.$}
\end{equation}
Moreover, if $1\leq k\leq m-1$, then
\begin{align}\label{A6}
|\nabla ^k u(x)|& \lesssim \chi_{\mathcal B_2^{n-1}} (x') \chi_{(L,
\infty)}(x_1)  \sum_{j=1}^{k+1} \int_{x_1}^\infty \int_{s_j}^\infty
\int_{s_{j+1}}^\infty \cdots \int_{s_{m-1}}^\infty f(s_m) \; ds_m
\cdots dt_j
\\
\cr & = \chi_{\mathcal B_2^{n-1}} (x') \chi_{(L, \infty)}(x_1)
 \sum_{j=1}^{k+1} \int_{x_1}^\infty  f(s_m) \frac{(s_m -
x_1)^{m-j}}{(m-j)!} \; ds_m   \nonumber
\\
\cr & \lesssim \chi_{\mathcal B_2^{n-1}} (x') \chi_{(L,
\infty)}(x_1)
 \sum_{j=1}^{k+1} \int_{x_1}^\infty  f(s_m) s_m ^{m-j} \; ds_m
  \nonumber
\\
\cr & \lesssim \chi_{\mathcal B_2^{n-1}} (x') \chi_{(L,
\infty)}(x_1) \int_{x_1}^\infty  f(s) s^{m-1} \; ds \qquad \hbox{for
a.e. $x\in \rn$}\nonumber
\end{align}
and, similarly,
\begin{equation}\label{A7}
|\nabla ^m u(x)| \lesssim \chi_{\mathcal B_2^{n-1}} (x') \chi_{(L,
\infty)}(x_1) \left (\int_{x_1}^\infty  f(s) s^{m-1} \; ds +
f(x_1)\right) \qquad \hbox{for a.e. $x\in \rn$}\,.
\end{equation}
Now, observe that, if a function $w\in \mathcal{M}_+(\rn)$ has the
form
\begin{equation*}
w(x)= g(x_1) \chi_{\mathcal{B}_N^{n-1}}(x')\qquad \hbox{for a.e.
$x\in \rn$},
\end{equation*}
for some $g\in \mathcal{M}_+ (\R)$ and $N>0$, then
\begin{align*}
%\mu_w (t) =
|\{ x\in \rn: w(x) >t\}| = \omega_{n-1} N^{n-1} |\{x_1\in \R:
g(x_1)>t\}|
%= \omega_{n-1} N^{n-1} \mu_g (t)
\qquad \hbox{for $t>0$\,,}
\end{align*}
whence
\begin{equation}\label{A8}
w^* (s) = g^* \left(\frac s{\omega_{n-1} N^{n-1}}\right)\qquad
\hbox{for $s\geq 0$\,.}
\end{equation}
From \eqref{A4}--\eqref{A8} we thus deduce that
\begin{equation}\label{A9}
u^*(s)\gtrsim \left ( \chi_{(2L, \infty)} (\, \cdot\, ) \int_{2 (\,
\cdot \,)}^\infty f(\tau) \tau^{m-1}\; d\tau \right)^*\left(\frac sc
\right) \qquad \hbox{for $s>0$}
\end{equation}
and
\begin{equation}\label{A10}
|\nabla ^m u|^*(s)\lesssim \int_{c\,s}^\infty f(r) r^{m-1}\; dr +
f^*(c\,s) \qquad \hbox{for $s>0$}\,,
\end{equation}
for some constant $c>0$. Note that here we have made use of
\eqref{sum*}. An application of inequality \eqref{E:10} yields, via
\eqref{A9}, \eqref{A10} and the boundedness of the dilation operator
in rearrangement-invariant spaces,
\begin{equation}\label{A11}
\left \| \chi_{(4L, \infty)} (s) \int_s^\infty f(r) r^{m-1} \; dr
\right \|_{S(0, \infty)} \lesssim \left \| \int_s^\infty f(r)
r^{m-1} \; dr \right \|_{X(0, \infty)} + \|f\|_{X(0, \infty)}
\end{equation}
for every $f\in\mathcal{M}_+(0, \infty)$ with bounded support. Now,
if $f$ is any function from $\mathcal{M}_+(0, \infty)$, one can
apply inequality
 \eqref{A11} with $f$ replaced by the function $f\chi _{(0,k)}$, for $k \in \N$,
 and pass to the limit as $k \to \infty$ to deduce \eqref{A11}.
 Note that  property (P3) of function norms
plays a role in this argument.
 %
% , and property \eqref{limit*}
 %of rearrangements play a role.
 Finally, choosing $L=\frac 18$ in
\eqref{A11} and applying Proposition \ref{P:1} tell us that
\begin{equation*}%\label{A12}
\|f^* \chi_{(1, \infty)} \|_{S(0, \infty)}\lesssim \|f\|_{X(0,
\infty)}
\end{equation*}
 for every $f\in \mathcal{M}_+(0, \infty)$, namely, \eqref{E:claim-2}.
\end{proof}

\begin{proof}[Proof of Corollary~\ref{T:supercrit}]
Under assumption \eqref{conv}, the function norm $\| \cdot
\|_{Z(0,1)}$, defined by \eqref{Z}, is equivalent to   $\| \cdot
\|_{L^\infty (0, 1)}$, up to multiplicative constants. Indeed,
\[
\|g\|_{Z'(0,1)}=\bigg\|s\sp{-1+\frac mn}\int_0\sp s
g\sp*(r)\,dr\bigg\|_{\widetilde X'(0,1)} \leq \left(\int_0\sp 1
g\sp*(r)\,dr \right)\;\big\| \chi _{(0,1)}(s) s^{-1+\frac mn}
\big\|_{X'(0,\infty)} \leq C\|g\|_{L\sp1(0,1)},
\]
for some constant $C$, and for every $g\in\M(0,1)$. This chain
establishes the embedding $L\sp1(0,1)\to Z'(0,1)$.  The converse
embedding follows from (P5). Thus $Z'(0,1)=L\sp{1}(0,1)$, whence
$Z(0,1)=L\sp{\infty}(0,1)$. The coincidence of the space $X^m_{\rm
opt}(\rn)$ with $(L^\infty \cap X)(\rn)$ then follows by the very
definition of the former.
\end{proof}

The last proof of this  section concerns
Theorem~\ref{T:reduction-higher}.

\begin{proof}[Proof of Theorem~\ref{T:reduction-higher}]
Suppose that conditions~\eqref{E:reduction2}
and~\eqref{E:reduction3} are in force. By
Theorem~\ref{T:main-higher},
\[
W\sp mX(\rn)\hra(X\sp m\cap X)(\rn).
\]
Thus, in order to prove \eqref{E:reduction1}, it suffices  to show
that
\begin{equation}\label{24a}
(X\sp m\cap X)(\rn)\hra Y(\rn).
\end{equation}
Observe that inequality~\eqref{E:reduction2} can be written in the
form
\begin{equation*}
\left\|\int_s^1 f(r) r^{-1+\frac{m}{n}} \;
dr\right\|_{\widetilde{Y}(0,1)} \lesssim \|f\|_{\widetilde{X}(0,1)}
\end{equation*}
for every non-increasing function $f : (0,1) \to [0, \infty)$, where
the function norms $\|\cdot \|_{\widetilde{X}(0,1)}$ and $\|\cdot
\|_{\widetilde{Y}(0,1)}$ are defined as in \eqref{tilde}.
 By the optimality of the space $Z (0,1)$ in inequality
\eqref{hardyopt}, one has that
\begin{equation*}%\label{E:embedding-X-to-S}
Z (0,1)\to \widetilde Y(0,1),
\end{equation*}
whence
\begin{equation}\label{22a}
\|g^*\chi_{(0,1)}\|_{Y(0,\infty)} \lesssim \|g^*\|_{ Z (0,1)}
%\|u\sp*\chi_{(0,1)}\|_{
%X\sp m(0,\infty)}\,.
\end{equation}
for any $g\in \mathcal {M}_+(0, \infty)$. Owing to \eqref{22a} and
\eqref{E:reduction3},
\begin{align*}
\|u\|_{(X\sp m\cap X)(\rn)} &= \|u\|_{X(\rn)}+\|u\|_{X\sp m(\rn)} =
\|u\sp*\|_{X(0,\infty)}+\|u\sp*\|_{ X\sp m(0,\infty)}    =
\|u\sp*\|_{X(0,\infty)}+\|u\sp*\|_{Z (0,1)}
\\
& \gtrsim
\|u\sp*\chi_{(1,\infty)}\|_{X(0,\infty)}+\|u\sp*\chi_{(0,1)}\|_{Y(0,\infty)}\geq
\|u^*\|_{Y(0,\infty)} = \|u\|_{Y(\rn)}
\end{align*}
for any $u\in (X^m\cap X)(\rn)$. Inequality \eqref{24a} is thus
established.
\par
Conversely,  assume that embedding \eqref{E:reduction1} holds. Owing
to the optimality of the rearrangement-invariant target norm
$X^m_{\rm opt}(\rn)$ in \eqref{T:main-higher},
\[
(X\sp m\cap X)(\rn)\hra Y(\rn)\,.
\]
%which means that, for some positive constant $C$ and every
%$u\in\M(\rn)$, we have
%\[
%\|u\|_{S(\rn)}\leq C(\|u\|_{X(\rn)}+\|u\|_{(X\cap X\sp m)(\rn)}.
%\]
By \eqref{E:definition-of-Y-opt-2-higher}, the latter embedding
implies that
\begin{equation}\label{E:necessity-embedding}
\|f\|_{Y(0,\infty)}\lesssim \|f\|_{X(0,\infty)}+\|f\sp*\|_{Z(0,1)}
\end{equation}
for any $f\in\M_+(0,\infty)$. In particular, applying this
inequality to functions $f$ of the form $g\sp*\chi_{(0,1)}$, and
making use of Proposition \ref{ZX} tell us that
\begin{align*}
\|g\sp*\chi_{(0,1)}\|_{Y(0,\infty)}  \lesssim
\|g\sp*\chi_{(0,1)}\|_{ X(0,\infty)}+\|g\sp*\|_{Z(0,1)}
 = \|g^*\|_{\widetilde X (0,1)} + \|g^*\|_{Z (0,1)}
\lesssim \|g\sp*\|_{ Z (0,1)}.
\end{align*}
In particular,
\[
\bigg\|\chi_{(0,1)}(s) \int _s^1 f(r) r^{-1+\frac mn}\, dr
\bigg\|_{Y (0,\infty)} \lesssim \bigg\|\int _s^1 f(r) r^{-1+\frac
mn}\, dr \bigg\|_{Z (0,1)}
\]
for every $f \in \mathcal M_+(0, \infty)$. Since, by
\eqref{hardyopt},
\[
\bigg\|\int _s^1 f(r) r^{-1+\frac mn}\, dr \bigg\|_{Z (0,1)}\lesssim
\|f\|_{ \widetilde X(0,1)} = \|\chi_{(0,1)}f\|_{ X(0,\infty)}
\]
for any non-increasing function $f : (0, \infty ) \to [0, \infty )$,
inequality \eqref{E:reduction2} follows. On the other hand, on
applying~\eqref{E:necessity-embedding} to functions of the form
$f\sp*\chi_{(1,\infty)}$, we obtain
\begin{align*}
\|f\sp*\chi_{(1,\infty)}\|_{Y(0,\infty)} & \lesssim
\|f\sp*\chi_{(1,\infty)}\|_{ X(0,\infty)}+f\sp*(1^-)\|1\|_{ Z (0,1)}
\\ & \lesssim \|f\sp*\chi_{(1,\infty)}\|_{
X(0,\infty)}+f\sp*(1^-)\|\chi_{(0,1)}\|_{ X(0,\infty)} \lesssim
\|f\|_{X(0,\infty)},
\end{align*}
namely~\eqref{E:reduction3}.
\end{proof}

%%%%%%%%%%%%%%%%%%%%%%%% proof of Theorem A %%%%%%%%%%%%%%%%%%%%%%%%%%%%%%%%%%%

\section{Proofs of Theorems \ref{T:orlicz2}, \ref{T:orlicz1} and \ref{T:application-to-lorentz-spaces}}

The proofs of our results about Orlicz-Sobolev embeddings require a
couple of preliminary lemmas.

\begin{lemma} \label{lemma1}
Let $F$ and $G$ be Young functions. Assume that there exist
constants $t_0>0$ and $c>0$ such that
\begin{equation*}
F(t)\leq G(c\,t) \qquad \hbox{\rm if} \;\; 0\leq t\leq t_0.
\end{equation*}
Let $L\geq 0$. Then
\begin{equation}\label{dic31}
\|f^*\|_{L^F(L, \infty)} \leq \max \Big\{{f^*(L)}/{t_0},
c\,\|f^*\|_{L^G(L , \infty)}\Big\}
\end{equation}
for every $f \in \M(0, \infty)$.
%
%g\in L^G(0, \infty)$ and $g^*(s)\leq 1$ for
%$s\in (\ell , \infty)$. Then $g\in L^F(\ell, \infty)$.
In particular,
\begin{equation}\label{dic30}
(L^G \cap L^\infty) (0, \infty) \to L^F(0, \infty).
\end{equation}
\end{lemma}

%%%%%%%%%%%%%%%%%%%%%%%%%%%%%%%%%%%   proof of Lemma 1   %%%%%%%%%%%%%%%%%%%

\begin{proof}
If $f^*(L)=0$, then $f^*=0$ in $(L, \infty)$, and \eqref{dic31}
holds trivially. Since the case when $f^*(L)=\infty$ is trivial as
well, we can in fact assume that $f^*(L) \in  (0,\infty)$. On
replacing $f$ with $\frac f{f^*(L)}$, we may suppose that
$f^*(L)=1$. Let
\begin{equation*}
\lambda =\max\Big\{ 1/{t_0}, c\,\|f^*\|_{L^G(L, \infty)}\Big\}\, .
\end{equation*}
Then,
\begin{align*}
\int_L ^\infty F\left(\frac {f^*(t)}{\lambda} \right)\; dt \leq
\int_L^\infty G\left(\frac{c\, f^*(t)}{\lambda}\right) \; dt \leq
\int_L^\infty G\left(\frac{f^*(t)}{\|f^*\|_{L^G(L, \infty)}}\right)
\; dt \leq 1\, ,
\end{align*}
and hence
\begin{equation*}
\|f^*\|_{L^F(L, \infty)} \leq  \max\Big\{1/{t_0}, c
\,\|f^*\|_{L^G(L, \infty)}\Big\} \, .
\end{equation*}
%Hence, inequality \eqref{dic31} follows on replacing $g$ with
%$\dfrac{g}{g^*(\ell)}$.
\\
Inequality \eqref{dic30} is a consequence of \eqref{dic31}, applied
with $L=0$, since $\|f\|_{L^\infty(0,\infty)}=f^*(0)$, and
$$\|f\|_{(L^G\cap L^\infty) (0, \infty)}
\simeq \max \big\{ \|f^*\|_{L^G(0 , \infty)},
\|f\|_{L^\infty(0,\infty)}\big\}.$$
\end{proof}

\begin{lemma} \label{lemma2}
Let $F$ and $G$ be Young functions such that
\begin{equation*}
F \quad\hbox{dominates}\quad G \qquad \hbox{near infinity}.
\end{equation*}
Assume that the function $H$, defined as
\begin{equation*}
H(t)=
\begin{cases}
G(t) &\textup{if}\;\; t\in [0, 1]
\\   F(t) &\textup{if}\;\; t\in (1, \infty)\,,
\end{cases}
\end{equation*}
is a Young function. Then
\begin{equation}\label{april7}
\|f^*\|_{L^F(0, 1)}+ \|f^*\|_{L^G(0, \infty)} \simeq \|f^*\|_{L^H(0,
\infty)} \end{equation} for every
 $f \in \M(0, \infty)$.
\end{lemma}

%%%%%%%%%%%%%%%%%%%%%%%%%%%%%%%%%%%   proof of Lemma 2   %%%%%%%%%%%%%%%%%%%

\begin{proof}
Define the rearrangement-invariant function norm $\|\cdot\|_{X(0,
\infty)}$ as
\begin{equation*}
\|f\|_{X(0, \infty)} = \|f^*\|_{L^F(0, 1)} + \| f^*\|_{L^G(0,
\infty)} \,
\end{equation*}
for $f \in \M_+(0, \infty)$.
%Since both $\| \cdot \|_{X(0, \infty)}$ and $\| \cdot \|_{L^H(0,
%\infty)}$ are rearrangement invariant norms,
Thanks to \cite[Theorem 1.8, Chapter 1]{BS}, equation \eqref{april7}
will follow if we show that
\begin{equation*}
X(0, \infty)=L^H(0, \infty)
\end{equation*}
as a set equality.
%Also, it suffices to assume that $f^*(1)=1$, since $f\in X(0,
%\infty)$ or $f\in L^H(0, \infty)$ if and only if
%$\frac{f^*(s)}{f^*(1)} \in X(0, \infty)$ or $\frac{f^*(s)}{f^*(1)}
%\in L^H(0, \infty)$, respectively.
%\\
Assume first that $f\in X(0, \infty)$. We have that
\begin{align}\label{april8}
\|f^*\|_{L^H(0, \infty)} &\leq \|f^*\|_{L^H(0, 1)} + \|f^*\|_{L^H(1,
\infty)} \lesssim \|f^*\|_{L^F(0, 1)} +  \|f^*\|_{L^H(1, \infty)}\,
,
\end{align}
where the second inequality holds since $H$ is equivalent to $F$
near infinity. By Lemma \ref{lemma1}, $\|f^*\|_{L^H(1,
\infty)}<\infty$ if $f\in L^G(0, \infty)$. Thus, equation
\eqref{april8} implies that $f\in L^H(0, \infty)$. Suppose next that
$f\in L^H(0, \infty)$. Then
\begin{align}\label{2}
\|f^*\|_{L^F(0, 1)} + \|f^*\|_{L^G(0, \infty)}& \leq \|f^*\|_{L^F(0,
1)}+ \|f^*\|_{L^G(0, 1)} + \|f^*\|_{L^G(1, \infty)}
\\
& \lesssim \|f^*\|_{L^F(0, 1)} + \|f^*\|_{L^G(1, \infty)}\nonumber
\\
& \lesssim \|f^*\|_{L^H(0, 1)} + \|f^*\|_{L^G(1, \infty)}\nonumber\,
,
\end{align}
where the second inequality holds since $F$ dominates $G$ near
infinity, and the third one since $F$ and $H$ agree near infinity.
Since $f\in L^H(0, \infty)$, then $f\in L^G(1, \infty)$ by Lemma
\ref{lemma1}. Thus, the right-hand side of \eqref{2} is finite,
whence $f\in X(0, \infty)$.
\end{proof}

 %ADD PROOF OF Proposition \eqref{redorl}.

The main ingredients for a proof of
 Theorem \ref{T:orlicz1} are now at our disposal.

\begin{proof}[Proof of Theorem \ref{T:orlicz1}]
It suffices to show that
\begin{equation}\label{10}
\|f\|_{(L^A)^m_{\rm opt} (0, \infty)}\simeq \|v\, f^*\|_{L^E(0,
\infty)}
\end{equation}
for $f \in\M_+(0,\infty)$. The sharp Orlicz-Sobolev embedding
theorem on domains with finite measure asserts that the optimal
rearrangement-invariant function norm $\|\cdot\|_{Z(0,1)}$ in
inequality \eqref{hardyopt}, defined as in \eqref{Z} with $X(0,
\infty) = L^A (0, \infty)$, obeys
\begin{equation}
\label{april9} \|f\|_{Z(0,1)}= \| s^{-\frac mn} f^*(s)
\|_{L^{\widehat{A}}(0, 1)}
\end{equation}
for $f \in \M_+(0,1)$ \cite[Theorem 3.7]{Ci3} (see also \cite{Ci2}
for the case when $m=1$). Hence, owing to Theorem
\ref{T:main-higher},
\begin{equation*}
\|f\|_{(L^A)^m_{\rm opt} (0, \infty)} = \| s^{-\frac mn} f^*(s)
\|_{L^{\widehat{A}}(0, 1)} + \|f^* \|_{L^A(0, \infty)}\,
\end{equation*}
for $f \in \M_+(0, \infty)$. Thus, equation \eqref{10} will follow
if we show that
\begin{equation}\label{11}
\|s^{-\frac mn} f^*(s)\|_{L^{\widehat{A}}(0, 1)} + \|f^*\|_{L^A(0,
\infty)} \simeq \|v(s) f^*(s)\|_{L^E(0, \infty)}\,
\end{equation}
for $f \in \M_+(0, \infty)$. Since the two sides of \eqref{11}
define rearrangement-invariant function norms, by \cite[Theorem 1.8,
Chapter 1]{BS} it suffices to prove that the left-hand side of
\eqref{11} is finite if and only if the right-hand side is finite.
%On replacing, if necessary, $f^*(s)$ with $\frac{f^*(s)}{f^*(1)}$,
%it suffices to establish \eqref{11} under the assumption that
%$\frac{f^*(s)}{f^*(1)}\leq 1$.
Assume that the left-hand side of \eqref{11} is finite for some $f \in \M_+(0, \infty)$. %Then clearly
%$f^*(1)<\infty$.
We have that
\begin{align*}
\|v(s) f^*(s)\|_{L^E(0, \infty)} &\leq \|v(s) f^*(s)\|_{L^E(0, 1)} +
\|v(s) f^*(s)\|_{L^E(1, \infty)} \\ & = \|s^{-\frac mn}
f^*(s)\|_{L^E(0, 1)} + \| f^*\|_{L^E(1, \infty)}\nonumber
\\& \lesssim \|s^{-\frac mn} f^*(s)\|_{L^{\widehat{A}}(0, 1)} + \|f^*\|_{L^E(1, \infty)}\, ,\nonumber
\end{align*}
where the second inequality holds inasmuch as $E$ and $\widehat{A}$
are equivalent near infinity. By Lemma \ref{lemma1},
$\|f^*\|_{L^E(1, \infty)} < \infty$, since we are assuming that
$\|f^*\|_{L^A(1, \infty)}<\infty$. Therefore, $\|v(s)
f^*(s)\|_{L^E(0, \infty)} < \infty$.
\\
Suppose next that the right-hand side of \eqref{11} is finite. Then
%$f^*(1)<\infty$, and
\begin{align*}
\left\|s^{-\frac mn} f^*(s)\right\|_{L^{\widehat{A}}(0, 1)} +
\|f^*\|_{L^A(0, \infty)} & \leq \left\|s^{-\frac mn}
f^*(s)\right\|_{L^{\widehat{A}}(0, 1)} + \|f^*\|_{L^A(0, 1)} +
\|f^*\|_{L^A(1, \infty)}
\\
& \lesssim \left\|s^{-\frac mn} f^*(s)\right\|_{L^{\widehat{A}}(0,
1)} + \left\|s^{-\frac mn}f^*(s)\right\|_{L^{\widehat{A}}(0, 1)} +
\|v(s)f^*(s)\|_{L^A(1, \infty)}
\\
& \lesssim \|v(s) f^*(s)\|_{L^E(0, 1)} + \|v(s) f^*(s)\|_{L^A(1,
\infty)}\, ,
\end{align*}
where the second inequality holds by equation \eqref{april9} and
Proposition \ref{ZX},
%\begin{equation*}
%\Lambda ^B(s^{-\frac nm}) (0, 1)\to L^A(0, 1),
%\end{equation*}
and the last one since  $E$ and $\widehat{A}$ are equivalent near
infinity. Inasmuch as  $v\,f^*\in L^E(1, \infty)$, one has that
$v\,f^*\in L^A(1, \infty)$, by Lemma \ref{lemma1}. Thus,
$\|s^{-\frac mn} f^*(s)\|_{L^{\widehat{A}}(0, 1)} + \|f^*\|_{L^A(0,
\infty)} < \infty$. \iffalse
\smallskip
\noindent (ii) By the sharp Orlicz-Sobolev embedding theorem on
domains with finite measure \cite{Ci4}, Theorem \ref{T:main-higher}
now asserts that
\begin{equation*}
\|u\|_{(L^A)^m_{\rm opt} (\rn)} = \|  u^*(s) \|_{L^\infty(0, 1)} +
\|u^* \|_{L^A(0, \infty)}\, .
\end{equation*}
By Lemma \ref{lemma2},
$$ \|  u^*(s) \|_{L^\infty(0,
1)} + \|u^* \|_{L^A(0, \infty)} \approx \|u^* \|_{L^{A^\infty}(0,
\infty)}.$$ Hence, the conclusion  follows. \fi
\end{proof}

 Theorem \ref{T:orlicz2} relies upon a specialization of
 the reduction principle of Theorem \ref{T:reduction-higher} to Orlicz spaces.

\begin{proposition}[\textbf{Reduction principle for Orlicz spaces}] \label{redorl} Let $A$ and $B$ be Young functions. Then
\begin{equation*}
W\sp{m,A}(\rn)\hra L^B(\rn)
\end{equation*}
if and only if there exists a constant $C$ such that
\begin{equation}\label{E:reductionorl2}
\bigg\|\int _s^1 f(r) r^{-1+\frac mn}\, dr\bigg\|_{L^B (0,1)} \leq C
\| f\|_{ L^A (0,1)}
\end{equation}
for every  function $f\in \M (0, 1 )$, and
\begin{equation}\label{E:reductionorl3}
\hbox{$A$ dominates $B$ near $0$.}
\end{equation}
\end{proposition}

\begin{proof} By Theorem \ref{T:reduction-higher} and Remark \ref{fdecreasing},
it suffices to show that  condition
%\eqref{E:reductionorl2}--
\eqref{E:reductionorl3} is equivalent to
%\eqref{E:reductionorl2} coupled with
the inequality
\begin{equation}
\label{E:reductionorl4} \| f\|_{L^B(1, \infty)} \leq C \|f\|_{L^A
(0,\infty)}
\end{equation}
for some constant $C$, and for every non-increasing function $f: (0,
\infty ) \to [0, \infty)$.
 Assume that \eqref{E:reductionorl3} holds.
 Given any  function $f$ of this kind, we have that
\begin{align*}
\|f\|_{L^A(0,1)}& \geq f(1)/ A^{-1}(1).
%\|1\|_{L^A(0,1)} = cf(1)
\end{align*}
 Hence, owing to assumption
\eqref{E:reductionorl3}, Lemma \ref{lemma1} implies that
\begin{align*}
\| f\|_{L^B(1, \infty)}\lesssim \max\{f(1), \|f\|_{L^A (0,\infty)}\}
\lesssim \|f\|_{L^A (0,\infty)}.
\end{align*}
Inequality \eqref{E:reductionorl4} is thus established.
\\ Conversely, assume that inequality \eqref{E:reductionorl4} is in force. Fix $r>1$.
An application of this inequality with $f=\chi_{(0,r)}$ yields
$$
\|1\|_{L^B(1,r)} \leq C\, \|1\|_{L^A(0,r)} \qquad \hbox{for $r >
1$,}$$ namely
$$\frac 1{B^{-1}(\frac 1{r-1})} \leq \frac C {A^{-1}(\frac 1{r})} \qquad
\hbox{for $r > 1$.}
$$
Hence, \eqref{E:reductionorl3} follows.
\end{proof}

\begin{proof}[Proof of Theorem \ref{T:orlicz2}]
Since $A_{\rm opt}$ is equivalent to $A_{\frac nm}$ near infinity,
by \cite[Inequality (2.7)]{Ci4} inequality \eqref{E:reductionorl2}
holds with $B=A_{\rm opt}$. Moreover, since $A_{\rm opt}$ is
equivalent to $A$ near zero, condition \eqref{E:reductionorl3} is
satisfied with $B=A_{\rm opt}$. Proposition \ref{redorl} thus tells
us that embedding \eqref{33} holds.
\\ On the other hand, if $B$ is any Young function such that
$ W\sp{m,A}(\rn)\to L^B(\rn), $ then, by Proposition \ref{redorl},
conditions \eqref{E:reductionorl2} and \eqref{E:reductionorl3} are
fulfilled. The former ensures that $A _{\frac nm}$ dominates $B$
near infinity -- see \cite[Proof of Theorem 3.1]{Ci3}. The latter
tells us that $A$ dominates $B$ near $0$. Altogether, $A_{\rm opt}$
dominates $B$ globally, whence $L^{A _{\rm opt}}(\rn ) \to
L^B(\rn)$.
\end{proof}

\begin{proof}[Proof of Theorem \ref{T:application-to-lorentz-spaces}]
 One has that $(L\sp{p,q}(\rn))'=
L\sp{p',q'}(\mathbb{R}\sp n)$, up to equivalent norms -- see e.g.
~\cite[Chapter~4, Theorem~4.7]{BS}. Since  $\big(\widetilde
f\big)\sp*=f\sp*$ for every $f \in \M_+(0,1)$,
$$\|f\|_{{\widetilde {(L\sp{p,q})}}'(0,1)}\simeq
\|f\|_{ L\sp{p',q'}(0,1)}.$$ Thus, on denoting by
$\|\cdot\|_{Z(0,1)}$ the rearrangement-invariant function norm
associated with $\|\cdot \|_{ L\sp{p,q}(0,\infty)}$ as in \eqref{Z},
one has that
\[
{\|f\|_{Z'(0,1)}}\simeq  \|s\sp{\frac
mn}f\sp{**}(s)\|_{L\sp{p',q'}(0,1)}
\]
for  $f \in \M_+(0,1)$. Note that here there is some slight abuse of
notation, since the functionals $\|\cdot\|_{{\widetilde
{L\sp{p,q}}}(0,1)}$ and $\|\cdot\|_{{\widetilde {L\sp{p,q}}}'(0,1)}$
are, in general, only equivalent to  rearrangement-invariant
function norms.
 \iffalse

By~\cite[Theorem~10.2]{CPS}, applied to $\alpha=0$ and
$I(t)=t\sp{1-\frac mn}$, $t\in[0,1]$, and noting that the
assumptions of that theorem are satisfied, we get
\[
\|g\|_{Z'(0,1)}=\|t\sp{\frac{1}{p'}-\frac{1}{q'}+\frac
mn}g\sp{**}(t)\|_{L\sp{q'}(0,1)}=\|g\|_{L\sp{(r,q')}(0,1)},
\]
where $\frac1r=\frac1{p'}+\frac mn$. We shall now evaluate
$\|\cdot\|_{Z(0,1)}$. It is easy to check that $\|\cdot\|_{Z'(0,1)}$
is an r.i.~norm. Therefore, we have
$\|\cdot\|_{Z(0,1)}=\|\cdot\|_{Z''(0,1)}$.

If $p=q=1$, then $r=\frac{n}{m}$ and $q'=\infty$. Hence,
by~\cite[Theorem~9.5.1(i)]{FS}, we get
$L\sp{(r,q')}(\rn)(0,1)=L\sp{r,q'}(0,1)$, and thus,
by~\cite[Theorem~9.6.2(i)]{FS}, we obtain
\[
\|g\|_{Z(0,1)}=\|g\|_{(L\sp{r,q'})'(0,1)}=\|g\|_{L\sp{\frac{n}{n-m},1}(0,1)}.
\]

If $1<p<\frac nm$ and $1\leq q\leq\infty$, then
$r=\frac{np}{np-n+mp}$, since $\frac1r=1-(\frac1p-\frac mn)$.
Therefore, $1<r<\infty$, whence, using the same argument as above,
we get
\[
\|g\|_{Z(0,1)}=\|g\|_{L\sp{\frac{np}{n-mp},q}(0,1)}.
\]

If $p=\frac nm$ and $1<q\leq\infty$, then,
by~\cite[Theorem~9.6.14]{FS}, we have
\[
\|g\|_{Z(0,1)}=\|g\|_{L\sp{\infty,q;-1}(0,1)}.
\]

If $p=\frac nm$ and $q=1$, then $r=1$ and $q'=\infty$, hence
$L\sp{(r,q')}(0,1)=L\sp{(1,\infty)}(0,1)=L\sp1(0,1)$. Therefore,
\[
\|g\|_{Z(0,1)}=\|g\|_{L\sp{\infty}(0,1)}.
\]

Finally, assume that either $m>n$ or $m\leq n$ and $p>\frac nm$.
Then $r<1$. Consequently, by~\cite[Lemma~9.5.8]{FS}, we have
\[
\|g\|_{Z(0,1)}=\|g\|_{L\sp{\infty}(0,1)}.
\]

\fi
 As shown in \cite[Proof of Theorem~5.1]{CP-traces},
\begin{equation*}%\label{E:new-star}
\|f\|_{Z(0,1)}\simeq
\begin{cases}
\|s\sp{\frac1p-\frac mn-\frac 1q} f\sp*(s)\|_{L\sp{q}(0,1)}
&\textup{if}\;\; p=q= 1,\ \textup{or}\ 1<p<\frac nm\ \textup{and}\ 1\leq q\leq\infty,\\
\|s\sp{-\frac 1q}(\log\frac2s)\sp{-1} f\sp*(s)\|_{L\sp{q}(0,1)}
&\textup{if}\;\; p=\frac nm\ \textup{and}\ 1< q\leq\infty,\\
\|f\sp*\|_{L\sp{\infty}(0,1)} &\textup{if}\;\; p>\frac nm,\
\textup{or}\ p=\frac nm\ \textup{and}\ q=1, \ \textup{or}\ m>n
\end{cases}
\end{equation*}
for every $f\in \M_+(0,1)$. \iffalse

By~\eqref{optspace},
\begin{align*}
\|u\|_{(L^{p,q})\sp m_{\opt}(\rn)} &\simeq \|u\|_{(L^{p,q})\sp
m(\rn)}+\|u\|_{L^{p,q}(\rn)} =
\|u^*\|_{Z(0,1)}+\|u^*\|_{L^{p,q}(0,\infty)}
\\
& \simeq \|u^*\|_{Z(0,1)}+\|u^*\chi _{(0,1)}\|_{L^{p,q}(0,\infty)}+
\|u^*\chi _{(1, \infty)}\|_{L^{p,q}(0,\infty)} \lesssim
\|u^*\|_{Z(0,1)} + \|u^*\chi _{(1, \infty)}\|_{L^{p,q}(0,\infty)}.
\end{align*}
Note that the last inequality rests upon  Proposition \ref{ZX}.
 Hence, the conclusion follows, via an analogous argument
as in the proof of equation \eqref{11}.  TO BE CHECKED!

 VARIANT:

\fi By~\eqref{optspace},
\begin{align*}
\|u\|_{(L^{p,q})\sp m_{\opt}(\rn)} &\approx \|u\|_{(L^{p,q})\sp
m(\rn)}+\|u\|_{L^{p,q}(\rn)} =
\|u^*\|_{Z(0,1)}+\|u^*\|_{L^{p,q}(0,\infty)}
\\
& =
\|u^*\|_{Z(0,1)}+\|s^{\frac{1}{p}-\frac{1}{q}}u^*(s)\|_{L^{q}(0,\infty)}.
\end{align*}
Hence, the conclusion follows, via an analogous argument as in the
proof of equation \eqref{11}.

\end{proof}

 \noindent {\bf Acknowledgments.} This research was
partly supported by the Research Project of Italian Ministry of
University and Research (MIUR)  2012TC7588 ``Elliptic and parabolic
partial differential equations: geometric aspects, related
inequalities, and applications" 2012,  by Research Project of GNAMPA
of Italian INdAM (National Institute of High Mathematics)
U2016/000464 ``Propriet\`{a} quantitative e qualitative di soluzioni
di equazioni ellittiche e paraboliche'' 2016, and the grant
P201-13-14743S of the Grant Agency of the Czech Republic.

\end{document}

\bibitem{BS}
C.~Bennett and R.~Sharpley.
\newblock {\em Interpolation of operators}, volume 129 of {\em Pure and Applied
  Mathematics}.
\newblock Academic Press, Inc., Boston, MA, 1988.

\bibitem{Boyd:Pacific}
D.~W. Boyd.
\newblock Indices for the {O}rlicz spaces.
\newblock {\em Pacific J. Math.}, 38:315--323, 1971.

\bibitem{Cianchi:Indiana}
A.~Cianchi.
\newblock A sharp embedding theorem for {O}rlicz-{S}obolev spaces.
\newblock {\em Indiana Univ. Math. J.}, 45(1):39--65, 1996.

\bibitem{Cianchi:Comm}
A.~Cianchi.
\newblock Boundedness of solutions to variational problems under general growth
  conditions.
\newblock {\em Comm. Partial Differential Equations}, 22(9-10):1629--1646,
  1997.

\bibitem{Cianchi:Forum}
A.~Cianchi.
\newblock Higher-order {S}obolev and {P}oincar\'e inequalities in {O}rlicz
  spaces.
\newblock {\em Forum Math.}, 18(5):745--767, 2006.

\bibitem{CKP}
A.~Cianchi, R.~Kerman, and L.~Pick.
\newblock Boundary trace inequalities and rearrangements.
\newblock {\em J. Anal. Math.}, 105:241--265, 2008.

\bibitem{CianchiPick:AM}
A.~Cianchi and L.~Pick.
\newblock Sobolev embeddings into {BMO}, {VMO}, and {$L_\infty$}.
\newblock {\em Ark. Mat.}, 36(2):317--340, 1998.

\bibitem{CianchiPick:TAMS}
A.~Cianchi and L.~Pick.
\newblock Optimal {S}obolev trace embeddings.
\newblock {\em Trans. Amer. Math. Soc.}, 368(12):8349--8382, 2016.

\bibitem{CPStrace}
A.~Cianchi, L.~Pick, and L.~Slav{\'{\i}}kov{\'a}.
\newblock {S}obolev embeddings, rearrangement-invariant spaces and {F}rostman
  measures.
\newblock {\em Preprint}.

\bibitem{CPS}
A.~Cianchi, L.~Pick, and L.~Slav{\'{\i}}kov{\'a}.
\newblock Higher-order {S}obolev embeddings and isoperimetric inequalities.
\newblock {\em Adv. Math.}, 273:568--650, 2015.

\bibitem{CianchiRandolfi}
A.~Cianchi and M.~Randolfi.
\newblock On the modulus of continuity of weakly differentiable functions.
\newblock {\em Indiana Univ. Math. J.}, 60(6):1939--1973, 2011.

\bibitem{DT}
T.~K. Donaldson and N.~S. Trudinger.
\newblock Orlicz-{S}obolev spaces and imbedding theorems.
\newblock {\em J. Funct. Anal.}, 8:52--75, 1971.

\bibitem{EGO}
D.~E. Edmunds, P.~Gurka, and B.~Opic.
\newblock Double exponential integrability of convolution operators in
  generalized {L}orentz-{Z}ygmund spaces.
\newblock {\em Indiana Univ. Math. J.}, 44:19--43, 1995.

\bibitem{Eyring}
H.~J. Eyring.
\newblock Viscosity, plasticity, and diffusion as example of absolute reaction
  rates.
\newblock {\em J. Chemical Physics}, 4:283--291, 1936.

\bibitem{HMT}
J.~A. Hempel, G.~R. Morris, and N.~S. Trudinger.
\newblock On the sharpness of a limiting case of the {S}obolev imbedding
  theorem.
\newblock {\em Bull. Austral. Math. Soc.}, 3:369--373, 1970.

\bibitem{KP}
R.~Kerman and L.~Pick.
\newblock Optimal {S}obolev imbeddings.
\newblock {\em Forum Math.}, 18(4):535--570, 2006.

\bibitem{NAFSA98}
M.~Krbec and A.~Kufner, editors.
\newblock {\em Nonlinear analysis, function spaces and applications. {V}ol. 6}.
  Academy of Sciences of the Czech Republic, Mathematical Institute, Prague,
  1999.

\bibitem{Mazya}
V.~G. Maz'ya.
\newblock {\em Sobolev spaces, with applications to elliptic partial
  differential equations}.
\newblock Springer, Berlin, 2011.

\bibitem{Musil}
V.~Musil.
\newblock Optimal {O}rlicz domains in {S}obolev embeddings into {M}arcinkiewicz
  spaces.
\newblock {\em J. Funct. Anal.}, 270(7):2653--2690, 2016.

\bibitem{Pokhozhaev}
S.~I. Pohozhaev.
\newblock On the imbedding {S}obolev theorem for $pl=n$.
\newblock {\em Doklady Conference, Section Math. Moscow Power Inst.},
  165:158--170 (Russian), 1965.

\bibitem{Strichartz:Indiana}
R.~S. Strichartz.
\newblock A note on {T}rudinger's extension of {S}obolev's inequalities.
\newblock {\em Indiana Univ. Math. J.}, 21:841--842, 1971/72.

\bibitem{Stromberg:Indiana}
J.-O. Str{\"o}mberg.
\newblock Bounded mean oscillation with {O}rlicz norms and duality of {H}ardy
  spaces.
\newblock {\em Indiana Univ. Math. J.}, 28(3):511--544, 1979.

\bibitem{Talenti79}
G.~Talenti.
\newblock Nonlinear elliptic equations, rearrangements of functions and
  {O}rlicz spaces.
\newblock {\em Ann. Mat. Pura Appl.}, 120:159--184, 1979.

\bibitem{Talenti}
G.~Talenti.
\newblock {\em An embedding theorem. Partial differential equations and the
  calculus of variations, Vol. II}.
\newblock Birkh\"auser, Boston, MA, 1989.

\bibitem{Talenti90}
G.~Talenti.
\newblock Boundedness of minimizers.
\newblock {\em Hokkaido Math. J.}, 19:259--279, 1990.

\bibitem{Trudinger}
N.~S. Trudinger.
\newblock On imbeddings into {O}rlicz spaces and some applications.
\newblock {\em J. Math. Mech.}, 17:473--483, 1967.

We shall now derive a~lower estimate for
\[
\left\|
\int_t\sp{\infty}g\sp{*}(s)\chi_{(1,\infty)}(s)s\sp{\frac1n-1}\,ds
\right\|_{S(0,\infty)}.
\]
We have
\begin{align*}
\int_t\sp{\infty}g\sp{*}(s)\chi_{(1,\infty)}(s)s\sp{\frac1n-1}\,ds
&=
\int_t\sp{\infty}\sum_{i=1}\sp{k}a_i\chi_{(\frac12,b_i)}(s)s\sp{\frac1n-1}\chi_{(1,\infty)}(s)\,ds\\
&=
\int_t\sp{\infty}\sum_{i\in\{1,\dots,k\},b_i>1}a_i\chi_{(1,b_i)}(s)s\sp{\frac1n-1}\,ds\\
&= \chi_{(0,1)}(t) n
\sum_{i\in\{1,\dots,k\},b_i>1}a_i\left(b_i\sp{\frac1n}-1\right) + n
\sum_{i\in\{1,\dots,k\},b_i>1}a_i\chi_{(1,b_i)}(t)\left(b_i\sp{\frac1n}-t\sp{\frac1n}\right).
\end{align*}
So,
\begin{align*}
&\left\|
\int_t\sp{\infty}g\sp{*}(s)\chi_{(1,\infty)}(s)s\sp{\frac1n-1}\,ds
\right\|_{S(0,\infty)}\\
&\gtrsim
\sum_{i\in\{1,\dots,k\},b_i>1}a_i\left(b_i\sp{\frac1n}-1\right)\|\chi_{(0,1)}(t)\|_{S(0,\infty)}
+ \left\| \sum_{i\in\{1,\dots,k\},b_i>1}a_i\left(b_i\sp{\frac1n}-
t\sp{\frac1n}\right)\chi_{(1,b_i)}(t)\right)\|_{S(0,\infty)}.
\end{align*}

======================

In accordance with the definition given above,
$\|\cdot\|_{\widetilde Y(0,1)}$ will denote the local version of the
norm $\|\cdot\|_{Y(0,\infty)}$. We define the functional
$\|\cdot\|_{\widetilde Y\sp1(0,1)}$ by
\begin{equation}\label{E:definition-of-Y-opt}
\|g\|_{(\widetilde Y\sp 1)'(0,1)}=\|t\sp{\frac
1n}g\sp{**}(t)\|_{\widetilde Y'(0,1)}, \quad g\in\M(0,1).
\end{equation}
We also denote by $\|\cdot\|_{Y\sp1(0,\infty)}$ the
rearrangement-invariant function norm whose localized version is
$\|\cdot\|_{\widetilde Y\sp1(0,1)}$. That is,
\begin{equation}\label{E:definition-of-Y-opt-2}
\|f\|_{Y\sp1(0,\infty)} = \|f\sp*\chi_{(0,1)}\|_{\widetilde
Y\sp1(0,1)}
\end{equation}